\documentclass[12pt]{article}
\usepackage[utf8]{inputenc}
\usepackage{amsmath, amsthm, amssymb, color, mathbbol,url}
\makeatletter
\def\maketag@@@#1{\hbox{\m@th\normalfont\normalsize#1}}
\makeatother
\usepackage[shortlabels]{enumitem}
\usepackage{graphicx}
\usepackage{makecell}
\usepackage[ruled]{algorithm2e}
\usepackage{multicol}
\usepackage[misc,geometry]{ifsym} 

\oddsidemargin \evensidemargin
\usepackage{enumitem}
\setenumerate{leftmargin=*}
\allowdisplaybreaks[3]	%allows page breaks in align

\newtheorem{theorem}{Theorem}
\newtheorem{proposition}[theorem]{Proposition}
\newtheorem{lemma}[theorem]{Lemma}
\newtheorem{corollary}[theorem]{Corollary}
\newtheorem{definition}[theorem]{Definition}

\newtheorem{remark}[theorem]{Remark}
\newtheorem{example}[theorem]{Example}

\newtheorem{question}[theorem]{Question}

\newcommand{\RR}{\mathbb{R}}

\date{}
\title{\textbf{Moment Identifiability of Homoscedastic \\ Gaussian Mixtures}}
\author{Daniele Agostini, Carlos Am\'endola\textsuperscript{\Letter} and Kristian Ranestad}

\begin{document}
	\maketitle
	
\begin{abstract}
We consider the problem of identifying a mixture of Gaussian distributions with same unknown covariance matrix by their sequence of moments up to certain order. Our approach rests on studying the moment varieties obtained by taking special secants to the Gaussian moment varieties, defined by their natural polynomial parametrization in terms of the model parameters. When the order of the moments is at most three, we prove an analogue of the Alexander-Hirschowitz theorem classifying all cases of homoscedastic Gaussian mixtures that produce defective moment varieties. As a consequence, identifiability is determined when the number of mixed distributions is smaller than the dimension of the space. In the two component setting we provide a closed form solution for parameter recovery based on moments up to order four, while in the one dimensional case we interpret the rank estimation problem in terms of secant varieties of rational normal curves.
\end{abstract}

	
\section{Introduction}
	
In the context of algebraic statistics \cite{sullivant2018algebraic}, moments of probability distributions have recently been explored from an algebraic and geometric point of view \cite{discgauss,momvar,polytopes,localdirac}. The key point for this connection is that in many cases the sets of moments define algebraic varieties, hence called \textit{moment varieties}. In the case of moments of mixture distributions, there is a natural correspondence to secant varieties of the moment varieties. Studying geometric invariants such as their dimension reveals properties such as model identifiability. One of the main applications for statistical inference is in the context of the method of moments, which matches the distribution's moments to moment estimates obtained from a sample. 
	
Gaussian mixtures are a prominent statistical model with multiple applications (see \cite{Cthesis} and references therein). They are probability distributions on $\mathbb{R}^n$ with a density that is a convex combination of Gaussian densities:

\begin{equation}\label{eq:densitymixtureG}
\lambda_1f_{\mathcal{N}{(\mu_1,\Sigma_1)}}(x) + \dots + \lambda_k f_{\mathcal{N}{(\mu_k,\Sigma_k)}}(x)
\end{equation} 
where $\mu_1,\dots,\mu_k\in \mathbb{R}^n$ are the $k$ means,  $\Sigma_1,\dots,\Sigma_k \in \operatorname{Sym}^2(\mathbb{R}^n)$ are the covariance matrices, and the $0\leq \lambda_i \leq 1$ with $\lambda_1+\dots+\lambda_k=1$ are the mixture weights. 
	
The starting point is thus the Gaussian moment variety $\mathcal{G}_{n,d}$, as introduced in \cite{momvar}, whose points are the vectors of all moments of order at most $d$ of an $n$-dimensional Gaussian distribution. The moments corresponding to the mixture density \eqref{eq:densitymixtureG} form the secant variety $\operatorname{Sec}_k(\mathcal{G}_{n,d})$, and identifiability in this general setting was the focus of \cite{algident}.
	
In this work we study special families of Gaussian mixtures, called homoscedastic mixtures, where all the Gaussian components share the same covariance matrix. In other words, a \textit{homoscedastic} Gaussian mixture has a density of the form
\begin{equation}
\sum_{i=1}^k \lambda_i f_{\mathcal{N}{(\mu_i,\Sigma)}}(x)
\end{equation}
where the Gaussian probability densities $f_{\mathcal{N}_(\mu_i,\Sigma)}(x)$ have all different means $\mu_i$ and \emph{same} covariance matrix $\Sigma$. The moments, up to order $d$, of homoscedastic Gaussian mixtures are still polynomials in the parameters (the means and the covariance matrix), and form the moment variety $\operatorname{Sec}^H_k(\mathcal{G}_{n,d})$. This is a set of special $k$-secants inside the secant variety $\operatorname{Sec}_k(\mathcal{G}_{n,d})$. 
	
The main question we are concerned with is: when can a general homoscedastic $k$-mixture of $n$-dimensional Gaussians be identified by its moments of order $d$? More precisely, denote by $\Theta^H_{n,k}$ the parameter space of means, covariances and mixture weights for homoscedatic mixtures, and the moment map by 
\begin{equation}\label{eq:map}
   M_{n,k,d}:\Theta^H_{n,k}\to \operatorname{Sec}^H_k(\mathcal{G}_{n,d}).
\end{equation}
The mixture parameters of a point on the moment variety $\operatorname{Sec}^H_k(\mathcal{G}_{n,d})$ can be uniquely recovered if the fiber of the moment map \eqref{eq:map} is a singleton up to natural permutations of the parameters. If this happens for a general point on the moment variety, we say that the mixture is \textit{rationally identifiable} from its moments up to order $d$. If the fiber of a general point is finite, we say that we have \textit{algebraic identifiability}. The parameters are \textit{not identifiable} if the general fiber of the moment map has positive dimension.

If the dimension of the parameter space is larger than the dimension of the space of moments, then one may expect any moment to lie on the moment variety. Clearly the fiber of the moment map must have positive dimension and we cannot have identifiability. We therefore distinguish the unexpected cases: when the dimension of the moment variety is less than the dimension of both the parameter space and the moment space, then we say that the moment variety  $\operatorname{Sec}^H_k(\mathcal{G}_{n,d})$ is \textit{defective}. In particular, defectivity implies non-identifiability.

We illustrate with an example:
	
\begin{example}\label{ex:main}
Let $n=2$, $k=2$ and $d=3$. That is, we 
consider moments up to order three for the homoscedastic mixture
of two Gaussians in $\RR^2$. The Gaussian moment variety $\mathcal{G}_{2,3}$ is 5-dimensional with 2 parameters for the mean vector and 3 for the symmetric covariance matrix. The parameters for the homoscedastic mixture are two mean vectors $\mu_1= \begin{pmatrix} \mu_{11} \\ \mu_{12} \\ \end{pmatrix}$ and $\mu_2= \begin{pmatrix} \mu_{21} \\ \mu_{22} \\ \end{pmatrix}$, the common covariance $\Sigma = \begin{pmatrix} \sigma_{11} &\sigma_{12}\\ \sigma_{12} &\sigma_{22}\\ \end{pmatrix}$ and the mixture weight $\lambda$ of the first component, in total $2 \times 2 + 3 + 1 = 8$ parameters. On the other hand, there are 9 bivariate moments up to order 3. Explicitly, the map is:
		$$
		\begin{matrix}
		m_{10} & = &  \lambda \mu_{11} +(1-\lambda) \mu_{21} \\
		m_{01} & = &  \lambda \mu_{12} +(1-\lambda) \mu_{22} \\
		m_{20} & = &  \lambda (\mu_{11}^2+\sigma_{11})+(1-\lambda) (\mu_{21}^2+\sigma_{11}) \\
		m_{02} & = &  \lambda (\mu_{12}^2+\sigma_{22})+(1-\lambda) (\mu_{22}^2+\sigma_{22})  \\
		m_{11} & = &  \lambda (\mu_{11} \mu_{12}+\sigma_{12})
		+(1-\lambda) (\mu_{21} \mu_{22}+\sigma_{12}) \\
		m_{30} & = &  \lambda (\mu_{11}^3+3 \sigma_{11} \mu_{11})
		+(1-\lambda) (\mu_{21}^3+3 \sigma_{11} \mu_{21}) \\
		m_{03} & = &  \lambda (\mu_{12}^3+3 \sigma_{22} \mu_{12})
		+(1-\lambda) (\mu_{22}^3+3 \sigma_{22} \mu_{22}) \\
		m_{21} & = &  \lambda (\mu_{11}^2 \mu_{12}+\sigma_{11} \mu_{12}+2 \sigma_{12} \mu_{11})
		+(1-\lambda) (\mu_{21}^2 \mu_{22}+\sigma_{11} \mu_{22}+2 \sigma_{12} \mu_{21}) \\
		m_{12} & = &  \lambda (\mu_{11} \mu_{12}^2+\sigma_{22} \mu_{11}+2 \sigma_{12} \mu_{12})
		+(1-\lambda) (\mu_{21} \mu_{22}^2+\sigma_{22} \mu_{21}+2 \sigma_{12} \mu_{22}) \\
		\end{matrix}
		$$
Since there are more moments than parameters, one would expect that the mixture parameters can be recovered. However, the dimension of ${\rm Sec}^H_2(\mathcal{G}_{2,3})$ equals $7$. This is one less than the expected dimension of $8$. Therefore it is defective and there is no algebraic identifiability. This means that the method of moments is doomed to fail in this setting. However, if one measures moments up to order $d=4$, it is possible to uniquely recover the mixture parameters.
\end{example}

As is often observed \cite{discgauss,momvar,localdirac}, a change of coordinates to cumulants tends to yield simpler representations and faster computations. This is the case here and hence we also study the \textit{cumulant varieties} of the homoscedastic Gaussian mixtures. For Example \ref{ex:main} above, the moment variety in cumulant coordinates is simply the cone over a twisted cubic curve (see Example \ref{ex:cone}). This is not a coincidence, as is shown in Section \ref{sec:three}. 
	
Our main results, Theorems \ref{thm:degree3} and \ref{thm:classificationkn+1},  identify the defective homoscedastic moment varieties when $d=3$ and  show that the homoscedastic moment variety is not defective when $k\le n+1$. 
These are analogues of the Alexander-Hirschowitz theorem on secant-defective Veronese varieties \cite{AHoriginal}.

This paper is organized as follows. In Section 2 we present the connection between moments and cumulants. The moment varieties corresponding to homoscedastic secants are defined in Section 3. In Section 4 we give general algebraic identifiability considerations and do a careful analysis of the subcases $d=3$, $k=2$ and $n=1$. Finally, we conclude with a summary of results and list further research directions.
	
\section{Moments and Cumulants} \label{sec:two}
	
To get started, we make some remarks about moments and cumulants from an algebraic perspective. To a sufficiently integrable random variable $X$ on $\mathbb{R}^n$, associate its moments $m_{a_1,\dots,a_n}[X]$ and cumulants $\kappa_{a_1,\dots,a_n}[X]$ through the generating functions in $\mathbb{R}[\![u_1,\dots,u_n]\!]$:
\begin{equation}
M_X(u)=\sum_{(a_1,\dots,a_n)} m_{a_1,\dots,a_n}[X]\frac{u_1^{a_1}\dots u_n^{a_n}}{a_1!\dots a_n!},\quad  K_X(u)=\sum_{(a_1,\dots,a_n)} \kappa_{a_1,\dots,a_n}[X]\frac{u_1^{a_1}\dots u_n^{a_n}}{a_1!\dots a_n!}.
\end{equation}
The information obtained from moments is equivalent to that from cumulants, since they are obtained from one another through the simple transformations
\begin{equation}
M_X(u) = \exp(K_X(u)), \qquad K_X(u) = \log(M_X(u))
\end{equation}
which are well-defined, because the $0$-th moment is always one, whereas the $0$-th cumulant is always zero:
$m_{0}[X]=1,\kappa_0[X]=0$ for every random variable $X$. In particular, moments and cumulants take values in the affine hyperplanes $\mathbb{A}^M_n$ and $\mathbb{A}^K_n$ of $\mathbb{R}[\![u_1,\dots,u_n]\!]$ defined by
\begin{equation}
\mathbb{A}^M_n = \left\{  m_0 = 1 \right\}, \qquad \mathbb{A}^K_n = \left\{  \kappa_0 = 0 \right\}.
\end{equation} 
We call these hyperplanes the \emph{moment space} and the \emph{cumulant space}.
	
	Taking only moments up to order $d$, replace the power series ring $\mathbb{R}[\![ u_1,\dots,u_n  ]\!]$ with the truncated ring $\mathbb{R}[\![u_1,\dots,u_n ]\!]/(u_1,\dots,u_n)^{d+1}$, and everything goes through. In particular, there is an analogous definition of the affine hyperplanes $\mathbb{A}^M_{n,d}$ and $\mathbb{A}^K_{n,d}$ which we denote again by moment space and cumulant space.
	
	\begin{example}[Dirac distribution]
		Let $\mu=(\mu_1,\dots,\mu_n)$ in $\mathbb{R}^n$ be a point. The Dirac distribution $\delta_{\mu}$ with center $\mu$ on $\mathbb{R}^n$ is given by 
		\begin{equation}
		\int_{\mathbb{R}^n} f(x) \delta_\mu(x) := f(\mu).
		\end{equation}
		If $X$ is a random variable on $\mathbb{R}^n$ with this distribution, its moment-generating function is
		\begin{equation}
		M_X(u) = \mathbb{E}[e^{u^tX}] = e^{u^t\mu} = \sum_{(a_1,\dots,a_n)} \mu_1^{a_1}\dots \mu_n^{a_n} \frac{u_1^{a_1}\dots u_n^{a_n}}{a_1! \dots a_n!}.
		\end{equation}
		The moments of $X$ are monomials evaluated at $\mu$. On the other hand, for the cumulant generating function 
		\begin{equation}
		K_X(u) = \log M_X(u) = \log e^{u^t\mu} = u^t\mu = \mu_1u_1+\dots +\mu_nu_n,
		\end{equation}
		the linear cumulants coincide with the coordinates of $\mu$, and the higher order cumulants are all zero.
		
		This has an immediate translation into algebro-geometric terms: the parameter space for all Dirac distributions is the space $\mathbb{R}^n$, and the image of the moment map of degree $d$,  $M\colon \mathbb{R}^n \to \mathbb{A}^M_{n,d}$	is the affine $d$-th Veronese variety $V_{n,d} \subseteq \mathbb{A}^M_{n,d}$.  On the other hand, the image of the cumulant map $K\colon \mathbb{R}^n \to \mathbb{A}^K_{n,d}$ is the linear subspace given by $\{ \kappa_2 = \kappa_3 = \dots = \kappa_d = 0 \}$, where $\kappa_i$ is the degree $i$-part of an element in $\mathbb{A}^K_{n,d}$.
	\end{example}
	
	\begin{example}[Gaussian distribution]
		Let $\mu \in \mathbb{R}^n$ be a point, and $\Sigma \in \operatorname{Sym}^2\mathbb{R}^n$ an $n\times n$ symmetric and positive-definite matrix. The Gaussian distribution on $\mathbb{R}^n$ with mean $\mu$ and covariance matrix $\Sigma$ is given by the density
		\begin{equation}
		f_{(\mu,\Sigma)}(x) := \frac{1}{\sqrt{\det(2\pi \Sigma)}} e^{-\frac{1}{2}(x-\mu)^t \Sigma^{-1} (x-\mu)}.
		\end{equation}
		If $X\sim \mathcal{N}(\mu,\Sigma)$ is a Gaussian random variable with these parameters, its moment-generating function and cumulant-generating function are given by
		\begin{equation}
		M_X(u) = e^{u^t\mu + \frac{1}{2}u^t\Sigma u}, \qquad K_X(u) = u^t\mu + \frac{1}{2}u^t\Sigma u.
		\end{equation}
		The Gaussian moment variety $\mathcal{G}_{n,d}\subseteq \mathbb{A}^M_{n,d}$ consists of all Gaussian moments up to order $d$. Observe that the corresponding cumulant variety is given simply by the linear subspace $\{ \kappa_3 = \dots = \kappa_d = 0 \} \subseteq \mathbb{A}^K_{n,d}$.
	\end{example}
	
	While our focus is on Gaussian distributions, our approach applies to general location families that admit moment and cumulant varieties. We illustrate this with the next example.
	
	\begin{example}[Laplace distribution] \label{ex:Laplace}
		The (symmetric) multivariate Laplace distribution has a location parameter $\mu \in \RR^n$ and a covariance parameter $\Sigma$, a positive-definite $n\times n$ matrix. Its density function involves the modified Bessel function of the second kind (see \cite[Chapter 5]{LAPL}), but it can be defined via its simpler moment generating function:
		\begin{equation}
		M_X(u) = \frac{\exp(u^t \mu)}{1- \frac12 u^t \Sigma u}, \qquad K_X(u) = u^t\mu - \log\left( 1-\frac{1}{2}u^t\Sigma u\right)
		\end{equation}
		with radius of convergence such that $|u^t \Sigma u| < 2 $.
		
		Moments and cumulants up to order $d=3$ match with the Gaussian case. Also note that when $\Sigma = 0$, the Dirac moment generating function is recovered.
		However, when $d
		\geq 4$, the Laplace cumulants are no longer a linear space in the cumulant space. 
	\end{example}
	
The multiplicative structure of the power series ring $\mathbb{R}[\![ u_1,\dots,u_n ]\!]$ makes it particularly suitable to independence statements with respect to moments. Indeed, if $X,Y$ are two independent random variables on $\mathbb{R}^n$ then 
\begin{equation*}
M_{X+Y}(u) = \mathbb{E}[e^{u^t(X+Y)}] = \mathbb{E}[e^{u^tX}e^{u^tY}] = \mathbb{E}[u^tX] \cdot \mathbb{E}[u^tY] = M_X(u)\cdot M_Y(u).
\end{equation*}
With cumulants it is even simpler: it holds that
\begin{equation*}
K_{X+Y}(u)= \log(M_{X+Y}) = \log(M_X M_Y) = \log(M_X) + \log(M_Y) = K_X(u) + K_Y(u). 
\end{equation*}
	
The group of affine transformations $\operatorname{Aff}(\mathbb{R}^n)$ acts naturally on both moments and cumulants: indeed, for any $A\in GL(n,\mathbb{R})$ and $b\in\mathbb{R}^n$ and a random variable $X$ on $\mathbb{R}^n$, 
\begin{equation*}
M_{AX+b}(u) = M_{AX}(u)\cdot M_b(u) = \mathbb{E}[e^{u^tAX}]\mathbb{E}[e^{u^tb}] = e^{u^tb}\cdot \mathbb{E}[e^{(A^tu)^tX}] = e^{u^tb}\cdot M_X(A^tu)
\end{equation*}
and
\begin{equation*}
K_{AX+b}(u) = \log(M_{AX+b}(u)) = \log(e^{u^tb} M_X(A^tu)) = u^tb + K_X(A^tu).
\end{equation*}
In particular, note that translations correspond simply to translations in cumulant coordinates, whereas they induce a more complicated expression in moment coordinates. 
	
\section{Homoscedastic Secants} \label{sec:three}
	
When Karl Pearson introduced Gaussian mixtures to model subpopulations of crabs \cite{PEARSON}, he also proposed the \textit{method of moments} in order to estimate the parameters. The basic idea is to compute sample moments from observed data, and match them to the distribution's moments expressed in terms of the unknown parameters. The method of moments estimates are the parameters that solve these equations. This is a classical estimation method in statistics; a good survey is \cite{lindsayMOM}, and a recent `denoised' version for Gaussian mixtures is \cite{denoisedMOM}. 
	
The method of moments is very friendly for mixture models because computing moments of mixture densities is straightforward, since for every measurable function $g\colon \mathbb{R}^n \to  \mathbb{R}$
\begin{equation}
\int_{\mathbb{R}^n} g(x)\left(\sum_{i=1}^k\lambda_i f_{(\mu_i,\Sigma_i)}(x) \right) dx =\sum_{i=1}^k \lambda_i \int_{\mathbb{R}^n} g(x)f_{(\mu_i,\Sigma_i)}(x)dx,
\end{equation}
and thus the moments are just linear combinations of the corresponding Gaussian moments.
	
As hinted in the introduction, this discussion can be rephrased in geometric terms: let $\mathcal{G}_{n,d}\subseteq \mathbb{A}^M_{n,d}$ be the Gaussian moment variety on $\mathbb{R}^n$ of order $d$. Then the moments of mixtures of Gaussians are linear combinations of points in $\mathcal{G}_{n,d}$, so that their corresponding variety is the $k$-th secant variety $\operatorname{Sec}_k(\mathcal{G}_{n,d})$. 
	
The densities of \emph{homoscedastic} Gaussian mixtures, where the Gaussian components share a common covariance matrix, have the form:
\begin{equation}
\lambda_1f_{(\mu_1,\Sigma)}(x) + \dots + \lambda_k f_{(\mu_k,\Sigma)}(x)
\end{equation}
where the $\mu_i \in \mathbb{R}^n$ are the mean parameters, the $\Sigma \in \operatorname{Sym}^2 \mathbb{R}^n$ is the common covariance parameters, and the $\lambda_i \in \mathbb{R}$ with $\lambda_1+\dots+\lambda_k = 1$ are the mixture parameters. Thus, the parameter space for homoscedastic mixtures is
\begin{equation}
\Theta^H_{n,k} := (\mathbb{R}^n)^{\times k} \times \mathbb{R}^{k-1} \times \operatorname{Sym}^2\mathbb{R}^n = \{ (\mu_1,\dots,\mu_k),(\lambda_1,\dots,\lambda_k),\Sigma \,|\, \lambda_1+\dots+\lambda_k = 1 \},
\end{equation}
and it has dimension
\begin{equation}
\dim \Theta^H_{n,k} = nk + k-1 + \frac{n(n+1)}{2} = (n+1)\left( k + \frac{n}{2} \right) - 1.
\end{equation}
The moment map for homoscedastic mixtures is then an algebraic map
$$M_{n,k,d}: \Theta^H_{n,k}\to \mathbb{A}^M_{n,d}.$$
Points on the image, the moments of homoscedastic mixtures, are linear combinations of points in $\mathcal{G}_{n,d}\subseteq\mathbb{A}^M_{n,d}$ which share the same covariance matrix. 
\begin{definition}\label{defs}
The \emph{homoscedastic $k$-secant variety}, denoted $\operatorname{Sec}^H_k(\mathcal{G}_{n,d})$, is the image of the moment map $M_{n,k,d}$. The \emph{fiber dimension} $\Delta^H_{n,k,d}$ is the general fiber dimension of the map $M_{n,k,d}$,
\begin{equation}\label{homoscedasticfiber}
\Delta^H_{n,k,d} = \dim \Theta^H_{n,k} - \dim \operatorname{Sec}^H_k(\mathcal{G}_{n,d}).
\end{equation}
We say that $\operatorname{Sec}^H_k(\mathcal{G}_{n,d})$ is \emph{algebraically identifiable} if  $\Delta^H_{n,k,d}=0$.
\end{definition}
The feasibility of the method of moments is based on computing points on the fibers of the moment map $M_{n,k,d}$. Algebraic identifiability of $\operatorname{Sec}^H_k(\mathcal{G}_{n,d})$ means that a general homoscedastic Gaussian mixture in the 
homoscedastic $k$-secant variety is identifiable from its moments up to order $d$ in the sense that only finitely many Gaussian mixture distributions share the same moments up to order $d$, whereas we reserve the term \emph{rationally identifiable} if a general fiber consists of a single point, up to label swapping. In case the general fiber is not finite, then it is positive-dimensional, there is no identifiability of the parameters from the moments up to order $d$, and a higher order is needed for identifiability (cf. Remark \ref{rmk:ident} and \cite[Problem 17]{momvar}).
    
Since the dimension of $\operatorname{Sec}^H_k(\mathcal{G}_{n,d})$ is always bounded by the dimension of the ambient space $\mathbb{A}^M_{n,d}$, there is a simple estimate for the fiber dimension:
	
\begin{lemma}\label{lem:basicineq}
	For all $n,d,k$ it holds that
	\begin{equation} \label{basicineq}
	\Delta^H_{n,k,d} \geq  \max\left\{(n+1)\left( k + \frac{n}{2} \right)  - \binom{n+d}{d} , 0 \right \}.
	\end{equation}	
\end{lemma}
\begin{proof}
The moment space $\mathbb{A}^M_{n,d}$ is an affine hyperplane inside the vector space $\mathbb{R}[[ u_1,\dots,u_n ]]/(u_1,\dots,u_n)^{d+1}$, hence it has dimension
\begin{equation}
\dim \mathbb{A}^M_{n,d} = \dim \mathbb{R}[[u_1,\dots,u_n]]/(u_1,\dots,u_n)^{d+1} - 1 = \binom{n+d}{d} -1.
\end{equation}
Since $\operatorname{Sec}^H_k(\mathcal{G}_{n,d}) \subseteq \mathbb{A}^M_{n,d}$ note that 
	\begin{equation}
	\Delta^H_{n,k,d}  = \dim \Theta^H_{n,k} - \dim \operatorname{Sec}^H(\mathcal{G}_{n,d}) \geq \dim \Theta^H_{n,h} - \dim \mathbb{A}^M_{n,d}
	\end{equation}
	which is exactly the inequality in the statement.
	\end{proof}
	
We expect that in general situations the inequality \eqref{basicineq} is in fact an equality. Hence, define the \emph{defect} to be
\begin{equation}
\delta^H_{n,k,d} := \Delta^H_{n,k,d} - \max\left\{(n+1)\left( k + \frac{n}{2} \right)  - \binom{n+d}{d}  , 0 \right\}.
\end{equation}

We say that $\operatorname{Sec}^H_k(\mathcal{G}_{n,d})$ is \emph{defective} if $\delta^H_{n,k,d}>0$. As observed earlier, defectivity implies non-identifiability.

\subsection{Cumulant representation} \label{sec:cumhom}
	
Let us explore how homoscedastic secants become simpler in cumulant coordinates, and how this representation can be used to check identifiability.
	
First, rephrase the situation in terms of random variables: let $Z=Z_{\Sigma}$ be a Gaussian random variable with mean $0$ and covariance matrix $\Sigma$, and let $B=B_{(\mu_1,\dots,\mu_k),(\lambda_1,\dots,\lambda_k)}$ an independent random variable with distribution given by a mixture of Dirac distributions:
\begin{equation}\label{eq:densitymixtureD}
\lambda_1\delta_{\mu_1}(x)+\dots+\lambda_k\delta_{\mu_k}(x).
\end{equation} 
Then, the random variable $Z+B$ has density given by the homoscedastic mixture \eqref{eq:densitymixtureG}. Moreover, if $m=\mu_1\lambda_1+\dots+\mu_k\lambda_k$ is the mean of $B$, we write $B=A+m$, where $A$ is a centered mixture of Dirac distributions. 
	
One can compute cumulants of this random variable as follows:
\begin{equation}
K_{B+Z}(u) = K_{B}(u) + K_Z(u) = K_{A}(u)+m^t u + \frac{1}{2}u^t\Sigma u
\end{equation}
and this suggests to parametrize the homoscedastic secants in cumulant coordinates as follows:
\begin{equation}
K\colon \Theta_{n,k}^0 \times \mathbb{R}^n \times \operatorname{Sym}^2\mathbb{R}^n \to \mathbb{A}_{n,d}^K, \qquad (A,m,\Sigma) \mapsto K_A(u) + m^tu + \frac{1}{2}u^t\Sigma u
\end{equation}
where $\Theta_{n,k}^0$ parametrizes the centered mixtures of Dirac distributions
\begin{equation}
\Theta_{n,k}^0 = \{ (\mu_1,\dots,\mu_k),(\lambda_1,\dots,\lambda_k) \,|\, \lambda_1\mu_1+\dots+\lambda_k\mu_k=0, \lambda_1+\dots+\lambda_k=1 \}
\end{equation}
The \textit{cumulant homoscedastic secant variety} $\log(\operatorname{Sec}^H_{k}(\mathcal{G}_{n,d}))$ is the image of the map $K$. Since in this variety, one can freely translate by the elements in $\mathbb{R}^n$ and $\operatorname{Sym}^2(\mathbb{R}^n)$, the first cumulants and the second cumulants can take any value. The constraints are in the cumulants of order three and higher. We summarize this discussion in the following lemma.
	
\begin{lemma}\label{lemma:cone}
		Let $\mathbb{A}^{K,3}_{n,d}$ be the space of cumulants of order at least three and at most $d$,  let
		\begin{equation}\label{phimap}
		\phi_{n,k,d}\colon \Theta^0_{n,k} \to \mathbb{A}^{K,3}_{n,d}, \qquad A\mapsto K_A(u)_3+K_A(u)_4+\dots+K_A(u)_d 
		\end{equation}
		be the cumulant map and let $C^0_{n,k,d}$ denote the closure $\overline{\phi_{n,k,d}(\Theta^0_{n,k})}$.  Then the cumulant homoscedastic secant variety $\log(\operatorname{Sec}^H_k(\mathcal{G}_{n,d}))$ is a cone over $C_{n,k,d}^0$.
	\end{lemma}
	
	\begin{remark}\label{remark:equationscumulants}
		In particular, the equations for the cumulant homoscedastic secant variety $\log(\operatorname{Sec}^H(\mathcal{G}_{n,d}))$ inside $\mathbb{A}^{K}_{n,d}$ are exactly the same as the equations for $C^0_{n,k,d}$ inside $\mathbb{A}^{K,3}_{n,d}$.
	\end{remark}
	
	The fiber dimension $\Delta_{n,k,d}^{H}$ can also be computed as the fiber dimension of the map $\phi_{n,k,d}$:
	
	\begin{lemma}\label{lemma:phi}
		The fiber dimension $\Delta^H_{n,k,d}$ is equal to the fiber dimension of $\phi_{n,k,d}$. In other words
		\begin{equation}
		\Delta^H_{n,k,d} = \dim \Theta^0_{n,k} - \dim C^0_{n,k,d} = (k-1)(n+1) - \dim C^0_{n,k,d}.
		\end{equation}	
	\end{lemma}
	\begin{proof}
		The fiber dimension $\Delta^H_{n,k,d}$ is the difference $\dim \Theta^H_{n,k} - \dim \log(\operatorname{Sec}^H_{k}(\mathcal{G}_{n,d}))$. We know that $\Theta^H_{n,k} \cong \Theta^0_{n,k} \times \mathbb{R}^n \times \operatorname{Sym}^2\mathbb{R}^n$. Moreover, Lemma \ref{lemma:cone} says that $\log(\operatorname{Sec}^H_{k}(\mathcal{G}_{n,d}))$ is the cone over $C^0_{n,k,d}$, which is precisely $\mathbb{R}^n\times \operatorname{Sym}^2\mathbb{R}^n \times C^0_{n,k,d}$, so that the first equality follows. For the second equality, the dimension of $\Theta^0_{n,k}$ can be computed as $nk+k-1-n = (n+1)(k-1)$.
	\end{proof}
	
	\begin{example}[$n=k=2 $ , $ d=3$]\label{ex:cone}
	Revisiting Example \ref{ex:main} from the introduction, we concluded that ${\rm Sec}^H_2(\mathcal{G}_{2,3}) \subset \mathbb{A}^M_{2,3} \cong \mathbb{A}^9$ is expected to be a hypersurface but it is actually of codimension 2. The ideal of ${\rm Sec}^H_2(\mathcal{G}_{2,3})$ is Cohen-Macaulay and determinantal (generated by the maximal minors of a $6 \times 5$-matrix) as described in \cite[Proposition 19]{momvar}. The homoscedastic cumulant variety $\log(\operatorname{Sec}^H_{2}(\mathcal{G}_{2,3}))$ is defined by the vanishing of the $2\times 2$ minors of $$\begin{pmatrix}k_{30}&k_{21}&k_{12}\\ k_{21}&k_{12}&k_{03}\\ \end{pmatrix}.$$ Note that indeed the first and second order cumulants $k_{10},k_{01},k_{20},k_{11},k_{22}$ do not appear in the equations above, so that the cumulant variety is the cone over the twisted cubic curve.
	\end{example}

	\begin{remark}\label{remark:estimationcumulants}
		To estimate the mixture parameters from the cumulants it is enough to consider the map $\phi_{n,k,d}$ of Lemma \ref{lemma:cone}. Indeed, suppose that we have a homoscedastic mixture with parameters $(((\lambda_1,\dots,\lambda_k),(\mu_1,\dots,\mu_k)),m,\Sigma) \in \Theta^0_{n,k}\times \mathbb{R}^n \times \operatorname{Sym}^2\mathbb{R}^n$ and suppose that its cumulants are known, so that in polynomial form 
		\begin{align}
		\begin{split}
		\kappa_1(u) &= m^t u \\
		\kappa_2(u) &= K_A(u)_2 + \frac{1}{2}u^t\Sigma u \\
		\kappa_3(u) &= K_A(u)_3 \\
		\kappa_4(u) &= K_A(u)_4 \\
		\vdots
		\end{split}.
		\end{align}
		Then to recover the parameters one can first try to recover the $\lambda_i$ and the $\mu_i$ from the cumulants of order three and higher, and then compute $m$ and $\Sigma$ from the cumulants of order one and two.
	\end{remark}
	
	\subsection{Veronese secants}
	
	We briefly observe that we can recast the above discussion in a way that makes apparent the connection to mixtures of Dirac distributions and, hence, to secants of Veronese varieties. To work with classical secant varieties, this time we work in moment coordinates. Now, every homoscedastic mixture is the distribution of a random variable of the form $Z+B$, where $B$ is a mixture of Dirac distributions and $Z$ is a centered Gaussian of covariance $\Sigma$, independent from $B$. Thus the moment generating function of this variable is
	\begin{equation}\label{eq:momentsZ+B}
	M_{Z+B}(u) = M_Z(u)M_B(u) = e^{\frac{1}{2}u^t\Sigma u} \cdot M_B(u).     
	\end{equation}
	Therefore the role of the covariance parameter is decoupled from the others: In particular, for $\Sigma=0$, one obtains the moment variety for mixtures of Dirac distributions. When restricting to moments $M(u)_d$ of degree at most $d$, this is precisely the $k$-secants to the Veronese variety $\operatorname{Sec}_k(\mathcal{V}_{n,d})$. The additive group $\operatorname{Sym}^2\mathbb{R}^n$ acts on the moment space $\mathbb{A}^M_{n,d}$ by
	\begin{equation}
	\operatorname{Sym}^2\mathbb{R}^n \times \mathbb{A}^M_{n,d} \to \mathbb{A}^M_{n,d} , \qquad (\Sigma,M(u)_d) \mapsto e^{\frac{1}{2}u^t\Sigma u} \cdot M(u)_d
	\end{equation}
	and so \eqref{eq:momentsZ+B} says that $\operatorname{Sec}^H_k(\mathcal{G}_{n,d})$ is the union of all the orbits of the points in $\operatorname{Sec}_k(\mathcal{V}_{n,d})$ under this action.
	
	This is useful because we can exploit well-known results on secants of Veronese varieties to address identifiability. First, let $\Delta^{\mathcal{V}}_{n,k,d}$ denote the fiber dimension of the $k$-secants to the Veronese variety $\operatorname{Sec}_k(\mathcal{V}_{n,d}) \subseteq \mathbb{A}^M_{n,d}$: by definition, this is 
	\begin{equation}\label{veronesefiber}
	\Delta^{\mathcal{V}}_{n,k,d} := nk+k-1 - \dim \operatorname{Sec}_k(\mathcal{V}_{n,d}).
	\end{equation}
	A basic estimate for the dimension of $\operatorname{Sec}_k(\mathcal{V}_{n,d})$ is given by the dimension of the ambient space $\dim \mathbb{A}^{M}_{n,d} = \binom{n+d}{d}-1$, hence 
	\begin{equation}
	\Delta^{\mathcal{V}}_{n,k,d} \geq \max\left\{ (n+1)k - \binom{n+d}{d}, 0 \right\}
	\end{equation}
	so that we can define the defect for the $k$-th secant of the Veronese variety as
	\begin{equation}
	\delta^{\mathcal{V}}_{n,k,d} := \Delta^{\mathcal{V}}_{n,k,d} - \max\left\{ (n+1)k - \binom{n+d}{d}, 0 \right\}.
	\end{equation}
	This number was famously computed by Alexander and Hirschowitz \cite{AHoriginal}, see also \cite{BOonAH}:
	\begin{theorem}[Alexander-Hirschowitz]
		The defect for the  Veronese variety is always zero, except in the following exceptional cases 
		\begin{align}
		& d = 2, 2\leq k\leq n  & &\Delta^{\mathcal{V}}_{n,k,2} = \frac{k(k-1)}{2} \nonumber \\
		& n=2, d=4, k=5 & &\delta^{\mathcal{V}}_{2,5,4} = 1 \nonumber\\
		& n=3, d=4, k=9 & &\delta^{\mathcal{V}}_{3,9,4} = 1  \\
		& n=4, d=3, k=7 & &\delta^{\mathcal{V}}_{4,7,3} = 1 \nonumber\\
		& n=4, d=4, k=14 & &\delta^{\mathcal{V}}_{4,14,4} = 1  \nonumber
		\end{align}
	\end{theorem}
	
	Moreover, for a general point $M(u)\in \operatorname{Sec}_k(\mathcal{V}_{n,d})$, consider the closed subset of $\operatorname{Sym}^2\mathbb{R}^n$ given by
	\begin{equation}
	D(M) := \{ \Sigma \in \operatorname{Sym}^2\mathbb{R}^n \,|\, e^{\frac{1}{2}u^t\Sigma u} \cdot M(u) \in \operatorname{Sec}_k(\mathcal{V}_{n,d})  \}.
	\end{equation} 
	We have the following relation between the fiber dimensions  (\ref{homoscedasticfiber}) and ($\ref{veronesefiber}$): 
	\begin{proposition}\label{prop:propD}
		%With the previous notations, 
		It holds that
		\begin{equation}
		\Delta^H_{n,k,d} = \Delta^{\mathcal{V}}_{n,k,d} + \dim D(M) 
		\end{equation}
		where $M\in\operatorname{Sec}_k(\mathcal{V}_{n,d})$ is a general point.
	\end{proposition}
	\begin{proof}
		By the previous discussion, the moment map for homoscedastic mixtures factors as a composition of two surjective maps
		\begin{equation}
		\Theta^H_{n,k} \to \operatorname{Sym}^2(\mathbb{R}^n) \times \operatorname{Sec}_k(\mathcal{V}_{n,d}) \to \operatorname{Sec}^H_{k}(\mathcal{G}_{n,d}).
		\end{equation}
		Hence, the fiber dimension of the composite map is the sum of the fiber dimensions of the two factors. For the first one this is $\Delta^{\mathcal{V}}_{n,k,d}$, so it remains to consider the second. Denote the second factor by $\rho \colon \operatorname{Sym}^2(\mathbb{R}^n) \times \operatorname{Sec}_k(\mathcal{V}_{n,d}) \to \operatorname{Sec}^H_{k}(\mathcal{V}_{n,d})$ and let $(\Sigma_o,M_o(u)) \in  \operatorname{Sym}^2(\mathbb{R}^n) \times \operatorname{Sec}_k(\mathcal{V}_{n,d})$ be a general point. The fiber is
		\begin{align}
		\begin{split}
		\rho^{-1}(\rho(\Sigma_o,M_o(u)))  = & \left\{ (\Sigma,M(u)) \,|\, e^{\frac{1}{2}u^t\Sigma u}\cdot M(u) = e^{\frac{1}{2}u^t\Sigma_o u} \cdot M_o(u) \right\} \\
		= & \{ (\Sigma,M(u)) \,|\, M(u) = e^{\frac{1}{2}u^t(\Sigma_o-\Sigma)u} \cdot M_o(u)  \} \\
		\cong & \{ \Sigma \in \operatorname{Sym}^2(\mathbb{R}^n) \,|\, e^{\frac{1}{2}u^t(\Sigma_o-\Sigma)u} \cdot M_o(u) \in \operatorname{Sec}_k(\mathcal{V}_{n,d}) \} \\
		= & \Sigma_o  - \{ \Sigma' \in \operatorname{Sym}^2(\mathbb{R}^n) \,|\, e^{\frac{1}{2}u^t\Sigma'u} \cdot M_o(u) \in \operatorname{Sec}_k(\mathcal{V}_{n,d})  \} \cong D^K(M_o),
		\end{split}
		\end{align} 
		concluding the proof.
	\end{proof}
	
	\begin{remark}
		In the range $(n+1)\left(k + \frac{n}{2}\right)\leq \binom{n+d}{d}$ where we expect identifiability for homoscedastic Gaussian mixtures, we see that  $\Delta^H_{n,k,d}=\delta^H_{n,k,d}$, and Alexander-Hirschowitz says that $\Delta^{\mathcal{V}}_{n,k,d} = \delta^{\mathcal{V}}_{n,k,d} = 0$. Hence Proposition \ref{prop:propD} yields
		\begin{equation}
		\delta^{H}_{n,k,d} = \dim D(M)
		\end{equation}
	\end{remark}

	\section{Moment Identifiability}\label{sec:four}
	
	Now we start to determine identifiability in various cases. To do so, it is convenient to change notation slightly. Up to now, we have identified moments and cumulants with their corresponding generating functions. In the next sections, it is useful to identify the parameters with polynomials as well. We replace the location parameter $\mu = (\mu_1,\dots,\mu_n)$ with the corresponding linear polynomial $u^t\mu=\mu_1u_1+\dots+\mu_nu_n$ and we replace the covariance parameter $\Sigma$ with the quadric $\frac{1}{2}u^t\Sigma u$. Of course, the two representations are equivalent, but the polynomial formalism is better suited to the cumulant space and the moment space. In particular, the linear polynomials live in the dual vector space $V=\operatorname{Hom}(\mathbb{R}^n,\mathbb{R})$, whereas the quadratic polynomials live in $\operatorname{Sym}^2 V$.
	
	The next inequality reflects the fact that increasing the order of moments (or cumulants) measured results in better identifiability:
	
	\begin{lemma}\label{lemma:monotondelta}
		The fiber dimensions of general fibers of $M_{n,k,d}$ and $M_{n,k,d+1}$ satisfy:
		\begin{equation}
		\Delta^H_{n,k,d} \geq \Delta^H_{n,k,d+1}.
		\end{equation}
	\end{lemma}
	\begin{proof}
		By definition, the fiber dimension $\Delta^H_{n,k,d}$ is the dimension of a general nonempty fiber of the moment map $M_{n,k,d}\colon \Theta^H_{n,k,d} \to  \mathbb{A}^M_{n,d}$. However, this map is the composition of the map $M_{n,k,d+1}\colon \Theta^H_{n,k,d} \to \mathbb{A}^M_{n,d+1}$ and the projection map $\mathbb{A}^M_{n,d+1}\to \mathbb{A}^M_{n,d}$, that forgets the moments of order $d+1$, so the conclusion follows.
	\end{proof}
	
	\begin{remark}\label{rmk:ident}
	Since Gaussian mixtures are identifiable from finitely many moments (see e.g. \cite{momvar}), the sequence 
	$$\Delta^H_{n,k,1} \geq \Delta^H_{n,k,2} \geq \dots \geq \Delta^H_{n,k,d} \geq \Delta^H_{n,k,d+1} \geq \dots$$ 
	must stabilize at $0$ for some large enough $d$.
	\end{remark}
	
	The following observation is less trivial. It %it tells us that if the number $k$ of mixed distribution is less than the dimension $n$ of the space plus one, 
	allows a reduction to the case  $n=k-1$. 
	
	\begin{proposition}\label{prop:equalitynk-2}
		Suppose that $d\geq 3$ and $n\geq k-1$, then
		\begin{equation}
		\Delta^H_{n,k,d} = \Delta^H_{k-1,k,d}.
		\end{equation}
	\end{proposition}
	\begin{proof}
		Use Lemma \ref{lemma:phi}, which says that the fiber dimension $\Delta^H_{n,k,d}$ is equal to the fiber dimension of the map
		\begin{equation}
		\phi_{n,k,d}\colon \Theta^0_{n,k} \to \mathbb{A}_{n,d}^{K,3}.
		\end{equation}
		This dimension can be computed by looking at the differential of the map at a general point. The parameter space is defined  as
		\begin{equation}
		\Theta^0_{n,k}=\left\{  ((\lambda_1,\dots,\lambda_k),(L_1,\dots,L_k)) \in \mathbb{R}^{k}\times V^k \,|\, \sum_{i=1}^k \lambda_i=1, \, \sum_{i=1}^k \lambda_iL_i = 0 \right\}.
		\end{equation}
		Let $p=((\lambda_1,\dots,\lambda_k),(L_1,\dots,L_k)) \in \Theta^0_{n,k}$ be a general point. Then the tangent space to $\Theta^{0}_{n,k}$ at the point is given by
		\begin{equation*}
		T_p\Theta^0_{n,k} = \left\{ ((\varepsilon_1,\dots,\varepsilon_k),(H_1,\dots,H_k)) \in  \mathbb{R}^k\times V^k \,|\, \sum_{i=1}^k\varepsilon_i = 0, \, \sum_{i=1}^k (\varepsilon_iL_i + \lambda_iH_i) = 0 \right\}.
		\end{equation*}
		The fiber dimension of $\phi_{n,k,d}$ coincides with the dimension of the kernel of the differential $d\phi_{n,k,d}$ at the general point $p$. In particular, since the point is general and $n\geq k-1$, we can suppose that $L_i=u_i$ for $i=1,\dots,k-1$ and that all the $\lambda_i$ are nonzero. In particular $L_k$ is a linear combination of $u_1,\dots,u_{k-1}$. Now, we claim that if $((\varepsilon_1,\dots,\varepsilon_k),(H_1,\dots,H_k))$ is in the kernel of $d\phi_{n,k,d}$ then the only variables appearing in the $H_i$ are $u_1,\dots,u_{k-1}$. If this is true, then we are done, because the kernel of $d\phi_{n,k,d}$ coincides with the kernel of $d\phi_{k-1,k,d}$ at the point $((\lambda_1,\dots,\lambda_k),(L_1,\dots,L_k)) \in \Theta^0_{k-1,k}$
		
		To prove the claim, observe that the map is given by the cumulant functions $\phi_{n,k,d}=(\kappa_3,\kappa_4,\dots,\kappa_d)$, so the kernel of $d\phi_{n,k,d}$ equals the intersection of the kernels of the $d\kappa_i$ for $i=3,...,d$. Therefore it is enough to prove the analogous claim for the kernel of the differential $d\kappa_3$ of $\kappa_3$.  Since the first moment is zero by construction, the third cumulant coincides with the third moment
		\begin{equation}
		\kappa_3 = \lambda_1L_1^3+\dots+\lambda_kL_K^3.
		\end{equation}
		Hence the differential is the linear map 
		\begin{equation}
		d\kappa_{3,p} \colon T_p\Theta^0_{n,k} \to \mathbb{A}^{K,3}_{n,d}, \qquad
		((\varepsilon_1,\dots,\varepsilon_k),(H_1,\dots,H_k)) \mapsto \sum_{i=1}^k(3\lambda_iH_i+\varepsilon_iL_i)L_i^2
		\end{equation}
		and if $((\varepsilon_1,\dots,\varepsilon_k),(H_1,\dots,H_k))$ is in the kernel, then it must be that
		\begin{equation}
		\sum_{i=1}^k h_i L_i^2=0, \qquad \text{where } h_i = 3\lambda_iH_i+\varepsilon_iL_i.
		\end{equation}
		Since $\lambda_k\ne 0$, this is equivalent to $\sum_{i=1}^k h_i(\lambda_kL_i)^2 = 0$ and since $\lambda_1L_1+\dots+\lambda_kL_k=0$, we see that
		\begin{align*}
		\sum_{i=1}^k h_i(\lambda_kL_i)^2 &= \sum_{i=1}^{k-1}h_i(\lambda_kL_i)^2 + h_k(\lambda_kL_k)^2 =  \sum_{i=1}^{k-1}h_i(\lambda_kL_i)^2 + h_k\left(\sum_{i=1}^{k-1}\lambda_iL_i\right)^{2} \\
		& = \sum_{i=1}^{k-1}(\lambda_k^2h_i+\lambda_i^2h_k)L_i^2 + 2h_k \left( \sum_{1 \leq i < j \leq k-1} \lambda_i\lambda_j L_iL_j \right). 
		\end{align*}
		By assumption $L_i=u_i$ for $i=1,\dots,k-1$, so this last expression is equal to zero if and only if 
		\begin{equation}
		\sum_{i=1}^{k-1}(\lambda_k^2h_i + \lambda_i^2h_k) u_i^2 =- 2h_k \left( \sum_{1 \leq i < j \leq k-1} \lambda_i\lambda_j u_iu_j \right).
		\end{equation}
		If this is true, then $h_k$ uses only the variables $u_1,\dots,u_{k-1}$. Indeed, if some other variable, say $y$, appears in $h_k$ then on the right hand side there is the monomial $yu_1u_2$, while there is no such a monomial on the left hand side. Likewise, if the variable $y$ appears in one of the $h_i$ for $i=1,\dots,k-1$: then on the left hand side there would be a monomial of the form $yu_i^2$, while there is no such monomial on the right hand side.
		
		Hence, the $h_i$ are polynomials in the $u_1,\dots,u_k$, and, by definition of the $h_i$, it follows that the same holds for the $H_i$. This proves the claim and the result follows.
	\end{proof}
	
	%This is useful to characterize identifiability when $n\geq k-1$ in Theorem \ref{thm:classificationkn+1}. 
	
	\subsection{Moments up to order $d=3$}
	
	When $d=3$ we determine the defect ${\delta}^H_{n,k,3}$ and the fiber dimension ${\Delta}^H_{n,k,3}$ of the map $${\phi}_{n,k,3}:\Theta^0_{n,k}\to \mathbb{A}^{K,3}_{n,3}$$ for each $n$ and $k$, and use Lemma \ref{lemma:phi}. 
	When $d=3$, the space $\mathbb{A}^{K,3}_{n,3}$ is identified with the space $\operatorname{Sym}^3V$ of homogeneous polynomials of degree three, and as noted in the proof of Proposition \ref{prop:equalitynk-2}, the third cumulants coincide with the third moments, so that: %that the map $\phi_{n,k,3}$ of Lemma \ref{lemma:phi} is: 
	\begin{equation}
	\phi_{n,k,3}\colon \Theta^0_{n,k} \to \operatorname{Sym}^3V \qquad ((L_1,\dots,L_k),(\lambda_1,\dots,\lambda_k)) \mapsto \lambda_1L_1^3+\dots+\lambda_kL_k^3.
	\end{equation}
	We compute the closure $C^0_{n,k,3}$ of the image.
	\begin{lemma}\label{lemma:d3image}
		The set $C^0_{n,k,3}$ is the Zariski closure of 
		\begin{equation}
		\{ H_1(u)^3 + \dots + H_k(u)^3 \,|\, H_1(u),\dots,H_k(u) \in \mathbb{R}^n \text{ linearly dependent }  \}.
		\end{equation}	
	\end{lemma}
	\begin{proof}
		Recall that $$\Theta^0_{n,k} = \{ ((L_1,\dots,L_k),(\lambda_1,\dots,\lambda_k)) \in V^{k} \times \mathbb{R}^{k-1} \,|\, \lambda_1+\dots+\lambda_k =1,\,\, \lambda_1L_1+\dots+\lambda_kL_k = 0  \}.$$ To compute the Zariski closure, suppose that all the $\lambda_i$ are strictly positive, so that in particular we can write
		\begin{equation}
		L_k  = -\frac{\lambda_1}{\lambda_k}L_1 - \dots -\frac{\lambda_{k-1}}{\lambda_k}L_{k-1}. 
		\end{equation}
		Since cubic roots are well defined over $\mathbb{R}$, 
		\begin{small}
		\begin{align*}
		\lambda_1&L_1^3 + \dots +\lambda_kL_k^3  = \lambda_1L_1^3 + \dots +\lambda_{k-1}L_{k-1}^3 - \lambda_k\left(\frac{\lambda_1}{\lambda_k}L_1 + \dots +\frac{\lambda_{k-1}}{\lambda_k}L_{k-1}\right)^3 \\
	%	& = (\sqrt[3]{\lambda_1}L_1)^3 + \dots + (\sqrt[3]{\lambda_{k-1}}L_{k-1})^3 -\left[\left( \frac{\sqrt[3]{\lambda_1}}{\sqrt[3]{\lambda_{k}}} \right)^2 \cdot \sqrt[3]{\lambda_1} L_1+ \dots + \left(\frac{\sqrt[3]{\lambda_{k-1}}}{\sqrt[3]{\lambda_{k}}} \right)^2 \cdot \sqrt[3]{\lambda_{k-1}} L_{k-1}\right]^3 \\
	%	& = H_1^3 + \dots + H_{k-1}^3 -\left[\left( \frac{\sqrt[3]{\lambda_1}}{\sqrt[3]{\lambda_{k}}} \right)^2 \cdot H_1+ \dots + \left(\frac{\sqrt[3]{\lambda_{k-1}}}{\sqrt[3]{\lambda_{k}}} \right)^2 \cdot H_{k-1}\right]^3 \\
		& = H_1^3 + \dots + H_{k-1}^3 + H_k^3
		\end{align*}
		\end{small}
		where $H_i := \sqrt[3]{\lambda_i} L_i$ for $i=1,\dots,k-1$, and $H_k := -\sum_{i=1}^{k-1} \left(\frac{\sqrt[3]{\lambda_i}}{\sqrt[3]{\lambda_k}}\right)^2 H_i$, using the equality $\sqrt[3]{\lambda_k}\frac{\lambda_i}{\lambda_k}=\left(\frac{\sqrt[3]{\lambda_i}}{\sqrt[3]{\lambda_k}}\right)^2\sqrt[3]{\lambda_i}$.  In particular, this shows immediately that $\lambda_1L_1^3+\dots+\lambda_kL_k^3$ can be written as a sum of cubic powers of linearly dependent linear forms. 
		
		For the converse, let $H_1,\dots,H_k$ be linearly dependent linear forms.  For the Zariski closure, it suffices to assume that $H_k = -\beta_1H_{1}-\dots-\beta_{k-1}H_{k-1}$ for some general $\beta_1,\dots,\beta_{k-1} \in \mathbb{R}$ strictly positive. So we want to write 
		\begin{equation}\label{eq:eqbeta}
		\beta_ i  = \left(\frac{\sqrt[3]{\lambda_i}}{\sqrt[3]{\lambda_k}}\right)^2
		\end{equation}
		for some positive $\lambda_1,\dots,\lambda_{k}\in \mathbb{R}$ such that $\lambda_1+\dots+\lambda_k=1$. Given such $\lambda_i$, the above computations yields
		\begin{equation}
		H_1^3+\dots+H_k^3 = \lambda_1L_1^3+\dots+\lambda_kL_k^3,
		\end{equation}
		where $L_i = \frac{1}{\sqrt[3]{\lambda_i}}H_i$ for $i=1,\dots,k-1$ and $L_k = -\frac{\lambda_1}{\lambda_k}L_1-\dots-\frac{\lambda_{k-1}}{\lambda_k}L_{k-1}$, so that $\lambda_1L_1+\dots+\lambda_kL_k = 0$, as wanted.
		
		To conclude, it remains to show that the equations \eqref{eq:eqbeta} have a solution: these equations are equivalent to 
		\begin{equation}\label{eq:cuberoots}
		(\sqrt{\beta_i})^3 = \frac{\lambda_i}{1-\lambda_1-\dots-\lambda_{k-1}} \qquad  \text{ for } i=1,\dots,k-1.
		\end{equation}
		Observe that the square roots are well defined since $\beta_i>0$ for all $i=1,\dots,k-1$. Moreover, if $(\lambda_1,\dots,\lambda_{k-1})$ is a solution to \eqref{eq:cuberoots}, then it is easy to see that all the $\lambda_i$ must be strictly positive: indeed, since the $\beta_i$ are positive, $\lambda_i$ and $1-\lambda_1-\dots-\lambda_{k-1}$ have the same sign. Thus, if one of the $\lambda_i$ is negative, then all the $\lambda_i$ are negative, but then $1-\lambda_1-\dots-\lambda_{k-1}>0$ which is absurd. 
		
		Now, setting $b_i = \sqrt{\beta_i}^3$, rewrite the equations as the linear system
		\begin{equation}\label{eq:syst}
		\begin{pmatrix}
		1+b_1 & b_1 & b_1 & \dots & b_1 \\
		b_2 & 1+b_2 & b_2 & \dots & b_2 \\
		b_3 & b_3 & 1+b_3 & \dots & b_3 \\
		\vdots & \vdots & \vdots & \ddots & \vdots \\
		b_{k-1} & b_{k-1} & b_{k-1} & \dots & 1+b_{k-1}
		\end{pmatrix}
		\begin{pmatrix}
		\lambda_1 \\ \lambda_2 \\ \lambda_3 \\ \vdots \\ \lambda_{k-1}
		\end{pmatrix}
		=
		\begin{pmatrix}
		b_1 \\ b_2 \\ b_3 \\ \vdots \\ b_{k-1}
		\end{pmatrix}.
		\end{equation}
		The matrix determinant lemma gives that $\det(\mathrm{I}+ b\cdot \mathbb{1}^T) = 1 + \mathbb{1}^T b$ = $1 + b_1+\dots+b_{k-1}$, which is positive since the $\beta_i$ are positive. This means the system \eqref{eq:syst} has a unique solution.
	\end{proof}
	
	\begin{remark}
		The proof of Lemma \ref{lemma:d3image} actually gives more: indeed, it shows that the image of the positive part
		\begin{equation}
		\Theta^{0,+}_{n,k} = \{ ((L_1,\dots,L_k),(\lambda_1,\dots,\lambda_k)) \in \Theta_{n,d}^0 \,|\, \lambda_i > 0 \text{ for all } i=1,\dots,k \},
		\end{equation}
		which is the one relevant in statistics,
		coincides with the set of sums $\{ H_1(u)^3+\dots+H_k(u)^3\}$, where the $H_i$ are positively linearly dependent, meaning that there are coefficients $\beta_1,\dots,\beta_k>0$ such that
		\begin{equation}
		\beta_1H_1+ \dots + \beta_{k}H_{k} = 0.
		\end{equation}
	\end{remark}
	
	\begin{remark}\label{rmk:d3cone}
		The set of sums of cubes of $k$ dependent linear forms has a natural interpretation in terms of the projective Veronese variety: indeed consider the third Veronese embedding of $\mathbb{P}(V)=\mathbb{P}^{n-1}$:
		\begin{equation}
		v_{3}\colon  \mathbb{P}(V) \hookrightarrow \mathbb{P}(\operatorname{Sym}^3V), \qquad [L] \mapsto [L^3].
		\end{equation} 
		For each $(k-2)$-dimensional linear subspace $\Pi \subseteq \mathbb{P}^{n-1}$ let $\operatorname{Sec}_k(v_3(\Pi)) \subseteq \mathbb{P}(\operatorname{Sym}^3V)$ be the $k$-th secant variety of its image $v_3(\Pi)$. Then, by Lemma \ref{lemma:d3image}, the variety $C^0_{n,k,3}$ is the affine cone over the union of these secants:
		\begin{equation}
		C^0_{n,k,3} = \operatorname{Cone}\left(\overline{\bigcup_{\Pi\subseteq \mathbb{P}^{n-1}} \operatorname{Sec}^k(v_3(\Pi))}\right).
		\end{equation}  
	\end{remark} 
	
	We compute the dimension of this variety, dividing it in the cases $k\leq n+1$ and $k\geq n+1$:

	\begin{proposition}\label{prop:d3}
	\begin{enumerate}[i)]
	    \item \label{lemma:d3case1}
		If $k\geq n+1$, then 
		\begin{equation}
		\dim C^0_{n,k,3} = \min\left\{ kn , \binom{n+2}{3} \right\}
		\end{equation}
		except in the case $n=5,k=7$, where $\dim C^0_{5,7,3} = 34$.
	    \item \label{lemma:d3case2}
		If $k\leq n+1$, then 
		\begin{equation}
		\Delta^H_{n,4,3} = 2,  \qquad
		\Delta^H_{n,3,3} = 2, \qquad 
		\Delta^H_{n,2,3} = 1. 
		\end{equation}	
		and when $k\geq 5$,
		\begin{equation}
		\Delta^H_{n,k,3} = 0.
		\end{equation}
	\end{enumerate}
	\end{proposition}
	\begin{proof}
	i) Since $k\geq n+1$,  Remark \ref{rmk:d3cone} shows that $C^0_{n,k,3}$ is the cone over the $k$-th secant variety $\operatorname{Sec}_k(v_3(\mathbb{P}^{n-1}))$. The dimension of this variety is computed by the Alexander-Hirschowitz theorem, so that
	\begin{equation}
	\dim C^0_{n,k,3} = \min\left\{ kn,\binom{n+2}{3} \right\}
	\end{equation}
	with the single exception of $n=5,k=7$, where the dimension is one less than the expected, hence $\dim C^0_{5,7,3} = 34$.	
	
	ii) Since $k\leq n+1$, Proposition \ref{prop:equalitynk-2} shows that $\Delta^H_{n,k,3} = \Delta^H_{k-1,k,3}$. Hence, for $k=2,3,4$ we see directly from Table \ref{tab:defective3} that
		\begin{equation}
		\Delta^H_{3,4,3}=2, \qquad \Delta^H_{2,3,3} =2 ,\qquad \Delta^H_{1,2,3}=1.    
		\end{equation}
	For $k\geq 5$ instead, follow the proof of Proposition \ref{prop:equalitynk-2} and show that the differential of $\phi_{k-1,k,3}\colon \Theta^0_{k-1,k,3} \to \operatorname{Sym}^3 V$ at a general point is injective.  For this, consider the kernel of the differential at a point $p=((\lambda_1,\dots,\lambda_k),(L_1,\dots,L_k))$. It consists of elements $((\varepsilon_1,\dots,\varepsilon_k),(H_1,\dots,H_k)) \in \mathbb{R}^k\times V^k$ such that $\varepsilon_1+\dots+\varepsilon_k = 0, \varepsilon_1L_1+\dots+\varepsilon_kL_k + \lambda_1H_1+\dots+\lambda_k H_k = 0$ and
	\begin{equation}
	\sum_{i=1}^{k-1}\ell_i L_i^2 + 2h\left( \sum_{1\leq i<j\leq k-1} \lambda_i\lambda_j L_i L_j \right) = 0,
	\end{equation}
	where $\ell_i = \lambda_k^2(3\lambda_iH_i+\varepsilon_iL_i)+\lambda_i^2(3\lambda_kH_k+\varepsilon_kL_k)$ and $h=3\lambda_k H_k+\varepsilon_kL_k$. Now choose the specific point $p$ given by  $\lambda_i = \frac{1}{k}$ for each $i=1,\dots,k$, $L_i = u_i$ for $i=1,\dots,k-1$ and $L_k = -u_1-\dots-u_{k-1}$. Then the above equation becomes 
	\begin{equation}\label{eq:d3relationdifflast}
	\sum_{i=1}^{k-1}\ell_i u_i^2 +\frac{2}{k^2} \cdot h\cdot\left( \sum_{1\leq i<j\leq k-1} u_i u_j \right) = 0.
	\end{equation}
	Let us write $h=h_1u_1+\dots+h_{k-1}u_{k-1}$. Then, in \eqref{eq:d3relationdifflast},  the coefficient of $u_au_bu_c$ is $\frac{2}{k^2}(h_a+h_b+h_c)$ for all $1\leq a<b<c \leq k-1$. Hence
	\begin{equation}
	h_a+h_b+h_c = 0, \qquad \text{ for all } 1\leq a<b<c \leq k-1.
	\end{equation}
	Let $1\leq a < b < c< d\leq k-1$ be any four distinct indices between $1$ and $k-1$. Then the previous equations translate into the linear system
	\begin{equation}
	\begin{pmatrix} 
		1 & 1 & 1 & 0 \\
		1 & 1 & 0 & 1 \\
		1 & 0 & 1 & 1 \\
		0 & 1 & 1 & 1
	\end{pmatrix}
	\begin{pmatrix}
	h_a \\ h_b \\ h_c \\ h_d
	\end{pmatrix}
	= 0.
	\end{equation}
	The matrix appearing in the linear system is invertible, so $h_a=h_b=h_c=h_d=0$. Since this holds for an arbitrary choice of four distinct indices, it follows that $h=0$. Now, relation \eqref{eq:d3relationdifflast} tells us that $\sum_{i=1}^{k-1}u_i^2 \ell_i = 0$, but since $u_1^2,\dots,u_{k-1}^2$ form a complete intersection of quadrics, they do not have linear syzygies, which implies that $\ell_i=0$ for each $i$. From the definitions of $\ell_i$ and $h$, it follows that $3\lambda_iH_i+\varepsilon_iL_i=0$ for each $i$ but then the other two relations $\sum_i \varepsilon_i=0$ and $\sum_i (\lambda_iH_i+\varepsilon_iL_i)=0$ imply that $H_i=0,\varepsilon_i=0$ for all $i$, which is what was needed.
	\end{proof}
	
		Now we are ready for a complete classification of defectivity when $d=3$:
\newpage
	
	\begin{table}[h!tb]
    \begin{minipage}{.5\linewidth}
      \centering
        \begin{tabular}{ | c | c | c | c | c | c | c | c | c |} \hline  $n$ & $k$ & $d$ & par & $N$ & exp & $\dim$ & $\delta$ & $\Delta$\\ 
				\hline 1 & 1 & 3 & 2 & 3 & 2& 2 & 0 & 0 \\
				\hline 1 & 2 & 3 & 4 & 3 & 3& 3 & 0 & 1 \\
				%\hline 2 & 1 & 2 & 5 & 5 & 5& 5 & 0 & 0 \\
				\hline 2 & 2 & 3 & 8 & 9 & 8& 7 & 1 & 1 \\
				\hline 2 & 3 & 3 & 11 & 9 & 9& 9 & 0 & 2 \\
				%\hline 3 & 1 & 2 & 9 & 9 & 9& 9 & 0 & 0 \\
				\hline 3 & 2 & 3 & 13 & 19 & 13& 12 & 1 & 1 \\
				\hline 3 & 3 & 3 & 17 & 19 & 17& 15 & 2 & 2 \\
				\hline 3 & 4 & 3 &21 & 19 & 19& 19 & 0 & 2 \\
				\hline 4 & 2 & 3 & 19 & 34 & 19 & 18 & 1 & 1 \\
				\hline 4 & 3 & 3 & 24 & 34 & 24 & 22 & 2 & 2 \\
				\hline 4 & 4 & 3 & 29 & 34 & 29 & 27 & 2 & 2 \\
				\hline 4 & 5 & 3 & 34 & 34 & 34 & 34 & 0 & 0 \\
				%\hline 5 & 1 & 2 & 20 & 20 & 20 & 20 & 0 & 0 \\
				\hline 5 & 2 & 3 & 26 & 55 & 26 & 25 & 1 & 1 \\
				\hline 5 & 3 & 3 & 32 & 55 & 32 & 30 & 2 & 2 \\
				\hline 5 & 4 & 3 & 38 & 55 & 38 & 36 & 2 & 2 \\
				\hline 5 & 5 & 3 & 44 & 55 & 44 & 44 & 0 & 0 \\
				\hline 5 & 6 & 3 & 50 & 55 & 50 & 50 & 0 & 0 \\
				\hline 5 & 7 & 3 & 56 & 55 & 55 & 54 & 1 & 2 \\
				\hline \end{tabular}
    \end{minipage}%
    \begin{minipage}{.5\linewidth}
      \centering
        \begin{tabular}{ | c | c | c | c | c | c | c | c | c |} \hline  $n$ & $k$ & $d$ & par & $N$ & exp & $\dim$ & $\delta$ & $\Delta$\\ 
             	\hline 6 & 2 & 3 & 34 & 83 & 34 & 33 & 1 & 1 \\
				\hline 6 & 3 & 3 & 41 & 83 & 41 & 39 & 2 & 2 \\
				\hline 6 & 4 & 3 & 48 & 83 & 48 & 46 & 2 & 2 \\
				\hline 6 & 5 & 3 & 55 & 83 & 55 & 55 & 0 & 0 \\
				\hline 6 & 6 & 3 & 62 & 83 & 62 & 62 & 0 & 0 \\
				\hline 6 & 7 & 3 & 69 & 83 & 69 & 69 & 0 & 0 \\
				\hline 6 & 8 & 3 & 76 & 83 & 76 & 75 & 1 & 1 \\
				\hline 6 & 9 & 3 & 83 & 83 & 83 & 81 & 2 & 2 \\
				\hline 7 & 2 & 3 & 43 & 119 & 43 & 42 & 1 & 1 \\
				\hline 7 & 3 & 3 & 51 & 119 & 51 & 49 & 2 & 2 \\
				\hline 7 & 4 & 3 & 59 & 119 & 59 & 57 & 2 & 2 \\
				\hline 7 & 5 & 3 & 67 & 119 & 67 & 67 & 0 & 0 \\
				\hline 7 & 6 & 3 & 75 & 119 & 75 & 75 & 0 & 0 \\
				\hline 7 & 7 & 3 & 83 & 119 & 83 & 83 & 0 & 0 \\
				\hline 7 & 8 & 3 & 91 & 119 & 91 & 91 & 0 & 0 \\
				\hline 7 & 9 & 3 & 99 & 119 & 99 & 98 & 1 & 1 \\
				\hline 7 & 10 & 3 & 107 & 119 & 107 & 105 & 2 & 2 \\
				\hline 7 & 11 & 3 & 115 & 119 & 115 & 112 & 3 & 3 \\
				\hline 7 & 12 & 3 & 123 & 119 & 119 & 119 & 0 & 4 \\
				\hline
        \end{tabular}
    \end{minipage} 
    \caption{\label{tab:defective3}  All instances of defective varieties $\operatorname{Sec}^H_{k}(\mathcal{G}_{n,3})$ for $n=1,\dots, 7$ with $d=3$. The column \textit{`par'} denotes the number of parameters and \textit{`exp'} the expected dimension.} \medskip
\end{table}

	\begin{theorem}\label{thm:degree3}
		For $d=3$, the defect $\delta^H_{n,k,3}=0$ for any $k$ and $n$,  with the following exceptions:
		\begin{itemize}
		    \item $n\geq k$ and $k=2$, where $\delta^H_{n,2,3}=1$. 
			\item $n\geq k$ and $k=3,4$, where $\delta^H_{n,k,3}=2$.
			\item $n=5$ and $k=7$, where $\delta^H_{5,7,3}=1$. 
			\item $n\geq 4$ and $n+1 < k \leq \frac{n^2+2n+6}{6}$ where $\delta^H_{n,k,3} = k-n-1$.
			\item $n\geq 4$ and $\frac{n^2+2n+6}{6} \leq k < \frac{n^2+3n+2}{6}$ where $\delta^H_{n,k,3} = n\left( \frac{n^2+3n+2}{6} -k \right)$.
		\end{itemize}
	\end{theorem}
	
%	Table \ref{tab:defective3} records all instances $\operatorname{Sec}^H_{k}(\mathcal{G}_{n,3})$ for $n=1,\dots, 7$ and range of $k$ that covers all defective cases (`par' denotes the number of parameters and `exp' the expected dimension).
	
	\begin{proof}
		First consider the case when $n\geq k$: then Proposition \ref{prop:d3} \ref{lemma:d3case2} applies.
		It is straightforward to check that $\delta^H_{n,k,3} = \Delta^H_{n,k,d}$, from which the statement of the theorem follows.
		
		For the cases where $k\geq n+1$, start with the exceptional case $n=5,k=7$:  Proposition \ref{prop:d3}  \ref{lemma:d3case1} gives that $\dim \overline{\phi_{n,k,3}(\Theta^0_{5,7})} = 34$, and the Lemma \ref{lemma:phi}  yields $\Delta^H_{5,7,3} = 2$ and $\delta^H_{5,7,3} = 1$.	
		
		Now, consider the other cases: Proposition \ref{prop:d3} \ref{lemma:d3case1} gives that
		\begin{equation}
		\dim \overline{\phi_{n,k,3}(\Theta_{n,k})} = \min\left\{ nk , \binom{n+2}{3} \right\}
		\end{equation}
		and then Lemma \ref{lemma:phi} shows that 
		\begin{small}
		\begin{equation*}
		\Delta^H_{n,k,3} = (k-1)(n+1) -  \min\left\{ nk , \binom{n+2}{3} \right\} = \max \left\{ k-n-1, k-n-1 + n\left( k- \frac{n^2+3n+2}{6}\right)  \right\}
		\end{equation*}
		\end{small}
		so that
		\begin{small}
		\begin{equation*}
		\delta^H_{n,k,3} = \max \left\{ k-n-1, k-n-1 + n\left( k- \frac{n^2+3n+2}{6}\right)  \right\} - \max\left\{ 0, (n+1)\left( k - \frac{n^2+2n+6}{6} \right)  \right\}.
		\end{equation*}
		\end{small}
		Suppose first that $n=1,2,3$: this implies $k\geq n+1 \geq \frac{n^2+2n+6}{6} \geq \frac{n^2+3n+2}{6}$ so that
		\begin{equation}
		\delta^H_{n,k,3} = k-n-1 + n\left( k- \frac{n^2+3n+2}{6}\right) - (n+1)\left( k - \frac{n^2+2n+6}{6} \right) = 0.
		\end{equation} 
		Now, suppose that $n\geq 4$. Then  $5\leq n+1\leq \frac{n^2+2n+6}{6} \leq \frac{n^2+3n+2}{6}$ and there are three possibilities for $k$: if $k\geq \frac{n^2+3n+2}{6}$, then  
		\begin{equation}
		\delta^H_{n,k,3} = k-n-1 + n\left( k- \frac{n^2+3n+2}{6}\right) - (n+1)\left( k - \frac{n^2+2n+6}{6} \right) = 0.
		\end{equation}
		If instead $\frac{n^2+2n+6}{6}\leq k < \frac{n^2+3n+2}{6}$, then 
		\begin{equation}
		\delta^H_{n,k,3} = k-n-1-(n+1)\left( k-\frac{n^2+2n+6}{6} \right) = n\left( \frac{n^2+3n+2}{6} -k \right)
		\end{equation}
		which is strictly positive. Finally, if $n+1\leq k < \frac{n^2+2n+6}{6}$ the defect is
		\begin{equation}
		\delta^H_{n,k,3} = k-n-1,
		\end{equation}
		which is positive if and only if $k>n+1$. 
	\end{proof}
	
	As a consequence, identifiability can be characterized whenever $k\leq n+1$:
	
	\begin{theorem}\label{thm:classificationkn+1}
	 Suppose $k\leq n+1$. If $k\geq 5$ then a general homoscedastic mixture is algebraically identifiable from moments up to order $3$. If instead $k=2,3,4$ then a general homoscedastic mixture is algebraically identifiable from the moments up to order $d=4$.
	\end{theorem}
    \begin{proof}
    	When $k\geq 5$ this follows immediately from Theorem \ref{thm:degree3} and Lemma \ref{lemma:monotondelta}. If instead $k=2,3,4$, thanks to Proposition \ref{prop:equalitynk-2}, it is enough to set $n=k-1$ and check the first $d$ for which we have identifiability: these are a finite number of cases that can be done by direct computation (e.g. in \texttt{Macaulay2} \cite{M2}), and we find that such a $d$ is 4.
    \end{proof} 
	
	\subsection{Mixtures with $k=2$ components}
	
	When $k=2$ we characterize the rational identifiability as well. Since the case $d=3$ is already covered, consider only $d\geq 4$.
	
	\begin{theorem}\label{theorem:theoremk2}
		The homoscedastic secant $\operatorname{Sec}^H_{2}(\mathcal{G}_{n,4})$ is algebraically identifiable. If $d\geq 5$, the homoscedastic secant $\operatorname{Sec}^H_2(\mathcal{G}_{n,d})$ is also rationally identifiable. 
	\end{theorem}
	\begin{proof}
		By Lemma \ref{lemma:phi} and Remark \ref{remark:estimationcumulants}, it is enough to consider the parameter space given by $\Theta^0_{n,2} = \{ ((L_1,L_2),(\lambda_1,\lambda_2)) \,|\, \lambda_1+\lambda_2=1, \lambda_1L_1+\lambda_2L_2 = 0 \}$ and the map
		\begin{equation}
		\phi_{n,2,d}\colon \Theta^0_{n,2} \to C^0_{n,2,d} \subseteq \mathbb{A}^{K,3}_{n,d}. %\qquad ((L_1,L_2),(\lambda_1,\lambda_2)) \mapsto \sum_{i=3}^d[\log(\lambda_1e^{L_1}+\lambda_2e^{L_2})]_i
		\end{equation}	
		In order to compute the general fiber of this map, note that since $d\geq 4$, it follows from Theorem \ref{thm:classificationkn+1} and its proof that the map has finite fibers. Hence, it is enough to restrict a general fiber to the open subset $\lambda_2\ne 0$. There we may assume  $L_2 = -\frac{\lambda_1}{\lambda_2}L_1 = -\frac{\lambda_1}{1-\lambda_1}L_1$. We thus compute the fibers of the induced map
		\begin{equation}
		F_{n,2,d}\colon V \times (\mathbb{R}\setminus \{1\}) \to \mathbb{A}^{K,3}_{n,d}, \qquad (L,\lambda) \mapsto \phi_{n,2,d}\left((\lambda,1-\lambda),\left(L, -\frac{\lambda}{1-\lambda}L\right)\right).
		\end{equation}
		In explicit terms, this map is given by the terms from degree 3 to degree $d$ of the logarithm $\log(\lambda e^{L}+(1-\lambda)e^{-\frac{\lambda}{\lambda-1}L})$. A computation shows that the first terms are:
		\begin{equation*}
		\log(\lambda e^{L}+(1-\lambda)e^{-\frac{\lambda}{\lambda-1}L}) = f_3(\lambda)L^3 + f_4(\lambda)L^4 + f_5(\lambda)L^5 + \dots
		\end{equation*}
		\begin{equation}
		f_3(\lambda) = \frac{\lambda(1-\lambda)(1-2\lambda)}{6(1-\lambda)^3}, \qquad f_4(\lambda) = \frac{\lambda(1-\lambda)(1-6\lambda(1-\lambda))}{24(1-\lambda)^4},
		\end{equation} 
		\begin{equation*}
		f_5(\lambda) = \frac{\lambda(1-\lambda)(1-2\lambda)(1-12\lambda(1-\lambda))}{120(1-\lambda)^5}.    
		\end{equation*}
		Now suppose that $d=4$, and let $L\in  V$ and $ \lambda \in \mathbb{R}\setminus \{1\}$ be general elements. In fact, it is enough to assume $L\ne 0$ and $\lambda\ne 0,1,\frac{1}{2}$, so that $\kappa_3 = f_3(\lambda)L^3 \ne 0$. In order to compute the fiber of the point $(\kappa_3,\kappa_4) = F_{n,2,4}(L,\lambda)$, first observe that  $\kappa_3=f_3(\lambda_0)L_0^3 = (\sqrt[3]{f_3(\lambda_0)}L_0)^3$ and that the polynomial $L_0:=\sqrt[3]{f_3(\lambda)}L$ can be computed explicitly: from the expression
		\begin{equation}
		\kappa_3 = \kappa_{300..0}u_1^3+\kappa_{030..0}u_2^3+\dots+\kappa_{00..03}u_n^3 + (\text{ terms with mixed monomials })
		\end{equation}
		then one obtains
		\begin{equation}
		L_0 = \sqrt[3]{\kappa_{300..0}}\cdot u_1+\sqrt[3]{\kappa_{030..0}}\cdot u_2+\dots+\sqrt[3]{\kappa_{00..03}}\cdot u_n.
		\end{equation} 
		
		In particular, $L=f_3(\lambda)^{-\frac{1}{3}}L_0$, so that the equation $\kappa_4 = f_4(\lambda)L^4$ translates into $\frac{f_4(\lambda)}{f_3(\lambda)^{\frac{4}{3}}} = \frac{\kappa_4}{L_0^4}$. Observe that $a := \frac{\kappa_4}{L_0^4}$ is a constant that can be computed explicitly by comparing a single nonzero coefficient of $L_0^4$ with the corresponding coefficient of $\kappa_4$: for example, if $\sqrt[3]{\kappa_{300..0}} \ne 0$, then 
		\begin{equation}
		a = \frac{\kappa_{400..0}}{(\sqrt[3]{\kappa_{300..0}})^4}.
		\end{equation}
		Now, the equation $\frac{f_4(\lambda)}{f_3(\lambda)^{\frac{4}{3}}} = a$ is equivalent to $\frac{f_4(\lambda)^3}{f_3(\lambda)^4} = a^3$, or more explicitly
		\begin{equation} 
		\frac{3}{32} \cdot \frac{(1-6\lambda(1-\lambda))^3}{\lambda(1-\lambda)(1-4\lambda(1-\lambda))^2} = a^3.
		\end{equation}
		Note that this expression is invariant under exchanging $\lambda$ with $1-\lambda$, as is expected from the symmetry of the situation. Hence, set $\gamma := \lambda(1-\lambda)$ and rewrite this expression as 
		\begin{equation}\label{eq:equationa}
		\frac{3}{32}\cdot \frac{(1-6\gamma)^3}{\gamma(1-4\gamma)^2}  = a^3.  
		\end{equation}
		This is a cubic equation with three possible solutions for $\gamma$, which means there is no rational identifiability. In order to get such, consider also the cumulants $\kappa_5$ of order 5: this adds the data $\kappa_5$ and the condition $\kappa_5 = f_5(\lambda)L^5$. In the above notation  $L=f_3(\lambda)^{-\frac{1}{3}}L_0$, so that the condition $\kappa_5 = f_5(\lambda)L^5$ becomes $\frac{f_5(\lambda)}{f_3(\lambda)^{\frac{5}{3}}} = \frac{\kappa_5}{L_0^5}$. As before, we see that $b := \frac{\kappa_5}{L_0^5}$ is a constant that can be computed explicitly by comparing a single nonzero coefficient of $L_0^5$ with the corresponding coefficient of $\kappa_5$: for example, if $\sqrt[3]{\kappa_{300..0}} \ne 0$, then 
		\begin{equation}
		b = \frac{\kappa_{500..0}}{(\sqrt[3]{\kappa_{300..0}})^5}.
		\end{equation}
		Now, the equation $\frac{f_5(\lambda)}{f_3(\lambda)^{\frac{5}{3}}} = a$ is equivalent to $\frac{f_5(\lambda)^3}{f_3(\lambda)^5} = b^3$, or more explicitly, as above, with the substitution $\gamma = \lambda(1-\lambda)$, 
		\begin{equation}\label{eq:equationb}
		\frac{15}{128}\cdot \frac{(1-6\gamma)^5}{\gamma(1-\gamma)^3(1-12\gamma)} = b^3.
		\end{equation}
		Hence, rational identifiability is obtained if the two equations \eqref{eq:equationa} and \eqref{eq:equationb} have a unique common solution $\gamma$. This means that the map $\mathbb{R}\dashrightarrow \mathbb{R}^2, \gamma \mapsto (g(\gamma),h(\gamma))$ is generically injective. This map extends to 
		\begin{equation*}
		\mathbb{\mathbb{R}} \to \mathbb{P}^2, \, \, \left[ \frac{3}{32}(1-6\gamma)^3(1-\gamma)^3(1-12\gamma), \frac{15}{128}(1-6\gamma)^5(1-4\gamma)^2,\gamma(1-4\gamma)(1-\gamma)^3(1-12\gamma) \right],    
		\end{equation*}
		i.e. a map defined by polynomials of degree $7$.   It is generically injective if and only if the closure of its image is a plane curve of degree $7$. This can be verified with \texttt{Macaulay2} \cite{M2}: the resulting curve is given by the equation
		\begin{multline*}
		849346560x^5y^2-679477248x^4y^3-29491200x^5yz+2674483200x^4y^2z-2439217152x^3y^3z\\ +256000x^5z^2 +79744000x^4yz^2+2415168000x^3y^2z^2-2616192000x^2y^3z^2\\ +499500000x^2y^2z^3-406500000xy^3z^3+474609375y^3z^4 = 0. 
		\end{multline*}	\end{proof}
	
		\begin{figure}[ht]
    \centering
      \includegraphics[scale=0.35]{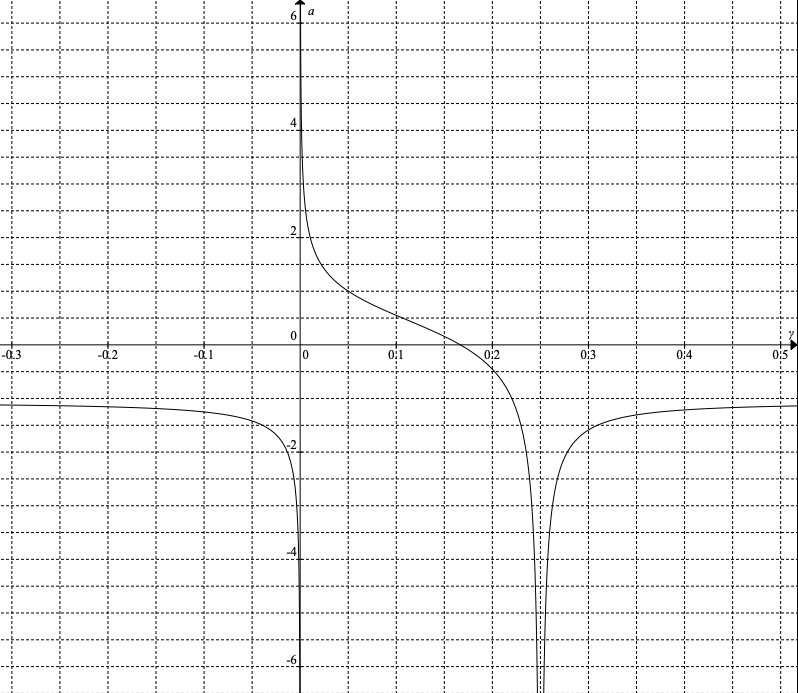}
       \caption{Plot of the real-valued function $a(\gamma)$ in \eqref{eq:reala}}
     \label{fig:plot}
      \end{figure}
       Even though there is no rational identifiability above when $d=4$, it is worth noting that in a purely statistical setting, $\gamma$ can be recovered uniquely, as seen below. 
      
	\begin{corollary}\label{cor:rident}
	    For $k=2$, the statistical mixture parameters can be recovered uniquely with moments up to order $d=4$. 
	    \begin{proof}
		This is equivalent to saying that the equation \eqref{eq:equationa} has a unique statistically relevant solution in $\gamma = \lambda (1- \lambda)$. Note that since $\lambda \in (0,1)\setminus \{\frac12 \}$, we have that $\gamma \in (0, \frac{1}{4})$. Consider the real valued function coming from \eqref{eq:equationa}:
		\begin{equation}\label{eq:reala}
		    a(\gamma)= \frac{\sqrt[3]{3}(1-6\gamma)}{2\sqrt[3]{4\gamma(1-4\gamma)^2}}.  
		\end{equation}
		Its derivative, $a'(\gamma) = -\frac{1}{2\sqrt[3]{36}\gamma(1-4\gamma)\sqrt[3]{4\gamma(1-4\gamma)^2}}$, is always negative for $0<\gamma<\frac{1}{4}$ so that the function $a(\gamma)$ is strictly decreasing and, in particular, injective in this statistically meaningful interval. 
		%A plot is presented in Figure \ref{fig:plot}.
		 The corresponding inverse is given by the cubic equation in $\gamma$
		 \begin{equation}\label{eq:cubic}
		     (256a^3+324)\gamma^3-(128a^3+162)\gamma^2+(16a^3+27)\gamma - \frac{3}{2}\, = \, 0.
		 \end{equation}
		 The discriminant of \eqref{eq:cubic} is $\Delta = -3072a^6(64a^3+81)$. It is zero precisely when $a=-\frac{3\sqrt[3]{3}}{4}$, which corresponds to the horizontal asymptote of $a$. If $a<-\frac{3\sqrt[3]{3}}{4}$, there are 3 real solutions, but one is negative and the other one is larger than $\frac{1}{4}$. The remaining solution is also the unique real solution when $a>-\frac{3\sqrt[3]{3}}{4}$, given explicitly by 
			\begin{equation}\label{eq:cubicsol}
			\gamma = \frac{4a^3}{3 \eta}
			+ \frac{\eta}  {3 (64 a^3 + 81)}+\frac{1}{6}. 
			\end{equation} 
		where $\eta=(-4096 a^9 - 10368 a^6 - 6561 a^3 + 9 \sqrt{262144 a^{15} + 995328 a^{12} + 1259712 a^9 + 531441 a^6} )^{\frac{1}{3}}$.
		\end{proof}
	\end{corollary}

	This proof gives an explicit algorithm to recover the parameters of a homoscedastic mixture of two Gaussians from the cumulants up to order four.
	
	\begin{algorithm}[H]
		\KwData{Data coming from a homoscedastic mixture of two Gaussian distributions.}
		\KwResult{The parameters $\lambda_1,\lambda_2,\mu_1,\mu_2,\Sigma$ of the mixture.}
		\Begin{
			Estimate the mean vector $\kappa_1$\;
			Estimate the covariance matrix $\kappa_2$\;
			Estimate the principal third cumulants $\kappa_{300..0},\kappa_{030..0},\dots,\kappa_{00..03}$\;
			For one of the principal third cumulant that is nonzero, estimate the corresponding fourth cumulant: in the following, we assume that $\kappa_{300..0}\ne 0 $, so that we estimate $\kappa_{400..0}$.\;
			Compute $a = \frac{\kappa_{400..0}}{(\sqrt[3]{\kappa_{300..0}})^4}$.\;
			Compute $\gamma$ as in \eqref{eq:cubicsol}\;
			Compute the two solutions $\lambda_1,\lambda_2$ of $\lambda(1-\lambda)=\gamma$\;
			Compute $\mu_1'=f_3(\lambda_1)^{-\frac{1}{3}}(\sqrt[3]{\kappa_{300..0}},\sqrt[3]{\kappa_{030..0}},\dots,\sqrt[3]{\kappa_{00..03}})$ and $\mu_2' = \frac{\lambda_1}{\lambda_2}\mu_1'$.\;
			Compute $\mu_1=\mu_1'+\kappa_1$ and $\mu_2 = \mu_2'+\kappa_1$\;
			Compute $\Sigma = 2(\kappa_2-(\lambda_1\mu_1+\lambda_2\mu_2)^t(\lambda_1\mu_1+\lambda_2\mu_2))$\;		
		}
		\caption{Recovery of parameters for a homoscedastic mixture of two Gaussians.}
	\end{algorithm}
	
	Observe that this algorithm needs all the cumulants of order one, all the cumulants of order two, $n$ cumulants of order three, and one cumulant of order four. Hence, it needs in total $n+\frac{n(n+1)}{2}+n+1$ cumulants.
	
\begin{remark}
	We have seen in Remark \ref{rmk:d3cone} that $\operatorname{Sec}^H_2(\mathcal{G}_{n,d})$ in cumulant coordinates is a cone over $C^0_{n,2,d} \subseteq \mathbb{A}^{K,3}_{n,d}$. Up to taking the Zariski closure, the proof of Theorem \ref{theorem:theoremk2} shows that $C^0_{n,2,d}$ is the image of the map
	\begin{equation}
	F_{n,2,d}\colon V\times \mathbb{R}\setminus\{1\} \to \mathbb{A}^{K,3}_{n,d}, \qquad (L,\lambda) \mapsto f_3(\lambda)L^3 + f_4(\lambda)L^4 + f_5(\lambda)L^5 + f_6(\lambda)L^6+ \dots
	\end{equation} 
	For $\lambda$  constant we get a projected $d$-th Veronese variety of $V$. If instead $L$ is constant, then we get a rational curve given by a linear combination of $(f_3(\lambda),f_4(\lambda),\dots,f_d(\lambda))$. 
%	In particular, we see that the $f_i(\lambda)$ are independent from the dimensional parameter $n$. 
%In particular, through a general point on $C^0_{n,2,d}$ there is a  the image of this map can be seen as a kind of scroll: indeed, it is a union of Veronese varieties connected by rational curves.
\end{remark}
	
\subsection{The univariate case $n=1$} \label{subsec:uni}
	
We use the standard notation $\sigma^2$ for the variance $\Sigma =( \sigma_{11})$ when $n=1$.
	
For $n=1$, the moment variety $\operatorname{Sec}_k^H(\mathcal{G}_{1,d})$ is never defective. The moment map 
$$M_{1,k,2k}:\Theta^H_{1,k} \to \mathbb{A}^M_{1,2k}  $$
is finite to one. In the statistics literature it is known that in the case of homoscedastic secants, one may recover mixture parameters from given moments (i.e. compute the fiber of the map above), with an algorithm closely related to the well-known Prony's method \cite{weissprony}. This procedure was introduced by Lindsay as an application of \textit{moment matrices} \cite{lindsaymommixt} and we briefly recall the algorithm here.
	
First, how does one recover the locations $\mu_i$ and weights $\lambda_i$ of the $k$ components of a Dirac mixture from $2k-1$ moments? This is known as the quadrature rule and it works as follows. Given the moment sequence $m=(m_1,m_2,\dots,m_{2k-1})$ one considers the polynomial resulting from the following $(k+1) \times (k+1)$ determinant
	
\begin{equation}
P_k(t) = \det \begin{pmatrix}
1 & m_1 & \dots & m_{k-1} & 1 \\
m_1 & m_2 & \dots & m_k & t \\
\vdots & & & \vdots & \vdots \\
m_k & m_{k+1} & \dots & m_{2k-1} & t^k \\
\end{pmatrix}.
\end{equation}
	
The $k$ roots $\mu_1, \mu_2, \dots, \mu_k$ of $P_k(t)$ are precisely the sought locations. This follows since the equations of the secant varieties of the rational normal curve are classically known to be given by the minors of the moment matrices. For a modern reference see \cite{landsbergottaviani}. 
	
Once the locations are known, the weights $\lambda_i$ are found by solving the $k \times k$ Vandermonde linear system
	
\begin{equation}
\begin{pmatrix}
1 & 1 & \dots & 1 \\
\mu_1 & \mu_2 & \dots & \mu_k \\
\vdots & & \vdots &  \\
\mu_1^{k-1} & \mu_2^{k-1} & \dots & \mu_k^{k-1} \\
\end{pmatrix} \begin{pmatrix}
\lambda_1 \\
\lambda_2 \\
\vdots \\
\lambda_k
\end{pmatrix}= \begin{pmatrix}
1 \\
m_1 \\
\vdots \\
m_{k-1}
\end{pmatrix}.
\end{equation}
	
Back to the Gaussian case, if we knew the value of the common variance $
\sigma^2$, we can reduce to the above instance. In terms of the Gaussian moment generating function:
	
\begin{equation}
e^{-\frac12 \sigma^2 u^2} M_X(u) = e^{\mu u}. 
\end{equation}
	
Hence, the Dirac moments $\tilde{m}$ on the right hand side are linear combinations of the Gaussian moments $m$. 
Explicitly, for $1 \leq j \leq 2k-1$
\begin{equation} \label{eq:transf}
\tilde{m}_j(\sigma) = \sum_{i=0}^{\lfloor j/2 \rfloor} \frac{j!}{(-2)^i i! (j-2i)!} m_{j-2i} \sigma^{2i}.
\end{equation}
Applying the quadrature rule to the vector $\tilde{m}=(\tilde{m}_1,\tilde{m}_2, \dots, \tilde{m}_{2k-1})$ would allow us to obtain the means $\mu_1, \mu_2, \dots, \mu_k$. 
	
However, $\sigma$ is unknown. To find an estimate for $\sigma$ we consider the first $2k$ moments $m = (m_1, m_2, \dots, m_{2k})$. If $\tilde{m}=(\tilde{m}_1,\tilde{m}_2, \dots, \tilde{m}_{2k})$ comes from a mixture of $k$ Dirac measures, then
\begin{equation} \label{eq:momdet}
D_k = \det \begin{pmatrix}
1 & \tilde{m}_1 & \dots & \tilde{m}_{k-1} & \tilde{m}_k \\
\tilde{m}_1 & \tilde{m}_2 & \dots & \tilde{m}_k & \tilde{m}_{k+1} \\
\vdots & & & \vdots & \vdots \\
\tilde{m}_k & \tilde{m}_{k+1} & \dots & \tilde{m}_{2k-1} & \tilde{m}_{2k} \\
\end{pmatrix} = 0.
\end{equation}
One thus treats $\sigma$ as a variable and substitutes expressions \eqref{eq:transf} into \eqref{eq:momdet}. This results in a polynomial $D_k(\sigma)$ of degree $\binom{k+1}{2}$ in $\sigma^2$ and the estimator $\hat{\sigma}^2$ is obtained as its smallest non-negative root \cite[Theorem 5B]{lindsaymommixt}. So the algebraic degree for estimating $\sigma^2$ is $\binom{k+1}{2}$.
With $\sigma^2$ specified, one proceeds as above. 
	
More generally, the discussion under \eqref{eq:momentsZ+B} shows that the moment variety $\operatorname{Sec}_k^H(\mathcal{G}_{1,d})$ with $k\leq d/2$ is a union
$$\operatorname{Sec}_k^H(\mathcal{G}_{1,d})=\bigcup_{\sigma} \operatorname{Sec}_k(V_{1,d}^{\sigma}), $$
where $V_{1,d}^{\sigma}$ is the translation of the moment curve $V_{1,d}$ by the variance $\sigma^2$ as defined by the Gaussian moments. The secant variety $\operatorname{Sec}_k(V_{1,d}^{\sigma})$ is defined for each $\sigma$ by the $(k+1)\times (k+1)$ minors of
	
\begin{equation} \label{eq:momhank}
M_{k,d} = \begin{pmatrix}
1 & \tilde{m}_1 & \dots & \tilde{m}_{d-k-1} & \tilde{m}_{d-k} \\
\tilde{m}_1 & \tilde{m}_2 & \dots & \tilde{m}_{d-k} & \tilde{m}_{d-k+1} \\
\vdots & & & \vdots & \vdots \\
\tilde{m}_k & \tilde{m}_{k+1} & \dots & \tilde{m}_{d-1} & \tilde{m}_{d} \\
\end{pmatrix}.
\end{equation}
As soon as the $k$-th secant variety of a smooth curve is not linear, the curve can be recovered as the singular locus of highest multiplicity in the secant variety. Therefore, since curves $V_{1,d}^{\sigma}$ are distinct, their $k$-th secant varieties are distinct as well, as long as the latter are not linear. In particular, since the variety  $\operatorname{Sec}_k(V_{1,d}^{\sigma})$ has dimension $2k-1$, it follows that the union $\operatorname{Sec}_k^H(\mathcal{G}_{1,d})$ has dimension $2k$.  Given the moments $m_i$ up to degree $d$ of a point on a homoscedastic $k$-secant, the $(k+1)\times (k+1)$ minors of $M_{k,d}$ are polynomials in $\sigma^2$ with a zero at the common variance. Given the variance, the means can be inferred as above.
	
When $d=2k+1$, then the variety  $\operatorname{Sec}_k^H(\mathcal{G}_{1,d})\subset \mathbb{A}^M_{1,2k+1}$ is a hypersurface, defined by the resultant of $(k+1)$-minors of $M_{k,d}$,  the polynomial obtained by elimination of $\sigma^2$ in the ideal defined by the 
$(k+1)\times (k+1)$ minors.  Denote this polynomial by $P_{2k+1}.$  It is a polynomial in $m_1,...,m_{2k+1}$ (or $\kappa_3,\kappa_4,\dots,\kappa_{2k+1}$). For example, $$P_3=\kappa_3 = 2m_1^3-3m_1m_2+m_3,$$
$$P_5= 108\kappa_3^6 - 32\kappa_3^2\kappa_4^3 + 36\kappa_3^3\kappa_4\kappa_5 - \kappa_4^2\kappa_5^2 + \kappa_3\kappa_5^3$$

\begin{proposition}  The polynomial $P_{2k+1}$ is homogeneous  of total degree 
$$\binom{k+2}{2}\binom{k+1}{2}$$
in the multigraded weights $\deg m_i=\deg \kappa_i = i$.
\end{proposition}
\begin{proof}  Let $$\mathbb{A}=\mathbb{A}^M_{1,2k+1}\times \mathbb{A}^1$$ where $\sigma$ is the last coordinate, and consider the projective closure $\mathbb{P}$ of $\mathbb{A}$.  Then the matrix (\ref{eq:momhank}) defines a map between vector bundles $E$ and $F$ on $\mathbb{A}$. The vector bundles $E$ and $F$  and the map extends to $\mathbb{P}$;   $E$ extends to a sum of line bundles  $\tilde E={\cal O}_\mathbb{P}\oplus {\cal O}_\mathbb{P}(-1)\oplus ...,{\cal O}_\mathbb{P}(-k)$, while F extends to a sum of line bundles $\tilde F={\cal O}_\mathbb{P}\oplus {\cal O}_\mathbb{P}(1)\oplus ...,{\cal O}_\mathbb{P}(k+1)$.  By the Thom-Porteous formula, \cite[Theorem 14.4]{Fulton},  the degree in $\mathbb{P}$ of the rank $k$ locus of the map is given by the Chern class $$c_2(\tilde F-\tilde E)= 2\binom{k+2}{2}\binom{k+1}{2} $$
since the Chern polynomials of $\tilde E$ and $\tilde F$ in are $$c(\tilde E)=(1-t)(1-2t)...(1-kt)$$
and $$c(\tilde F)=(1+t)(1+2t)...(1+(k+1)t).$$
This rank $k$ locus has codimension $2$ and its intersection with $\mathbb{A}$ is projected to the hypersurface defined by $P_{2k+1}$ in $\mathbb{A}^M_{1,2k+1}$.  The coordinate $\sigma$ appears only in even degree in the equations defining the rank $k$ locus, so the projection to $\mathbb{A}^M_{1,2k+1}$ is $2:1$, so the degree of $P_{2k+1}$ is half the degree of the rank $k$ locus.
\end{proof}

\begin{question} It would be interesting to understand better the structure of the polynomials $P_{2k+1}$, e.g. is there a closed form expression for all $k$? 
\end{question} 

%Under the weights $\deg m_i=\deg \kappa_i = i$ we see that $P_{2k+1}$ is an homogeneous polynomial.  What is its degree? Computations when $k=1,2,3,4$ show that the degree is \textcolor{red}{THIS SHOULD BE FIXED $k(k+1)(k+2)/2$} .  Does this formula hold for any $k$? 
If $P_{2k+1}$ vanishes on a the set $(m_1,...,m_{2k+1})$ of moments, and 
$P_{2l+1}$ does not vanish on $(m_1,...,m_{2l+1})$ for any $l<k$, then the moments lie on a homoscedastic $k$-secant but not on any $l$  secant for $l<k$. Therefore the polynomials $P_{2k+1}$ may be used to estimate the number of components in a homoscedastic Gaussian mixture (compare to the rank test proposed in \cite[Section 3.1]{lindsaymommixt} for the known variance case).
\section{Conclusion}
We have completely classified all defective cases for the moment varieties associated to homoscedastic Gaussian mixtures whenever $k<n+1$, $d=3$, $k=2$ or $n=1$. The question concerning a complete classification for all $n,d,k$ remains open, although our computations did not reveal any further defective examples. 

Our identifiability results also cover special structures in the covariance matrix, by Remark \ref{remark:estimationcumulants}. For example, a common mixture submodel involves isotropic Gaussians, which means that the covariance matrix is a scalar multiple of the identity, $\Sigma = \sigma I$. The $k$-means algorithm used in clustering can be interpreted as parameter estimation for a homoscedastic isotropic mixture of Gaussians. In \cite{hsu2013learning}, Hsu and Kakade consider the learning of mixtures of isotropic Gaussians from the moments up to order $d=3$ when $k \leq n+1$. They prove identifiability for the homoscedastic isotropic submodel (see \cite[Theorem 3.2]{anandkumar2014tensor}), and in order to solve the moment equations, they find orthogonal decompositions of the second and third order moment tensors. 

On the other hand, in \cite{lindsay1993multivariate} Lindsay and Basak proposed a `fast consistent' method of moments for homoscedastic Gaussian mixtures in the multivariate case, based on a `primary axis' to which the one-dimensional case presented in Section \ref{subsec:uni} is applied. This means that the method uses some moments of order $2k$. Knowing that in some cases there are explicit equations for secants of higher dimensional Veronese varieties \cite{landsbergottaviani}, an alternative method with minimal order based on these should be possible.

Finally, a similar approach can be made to study moment varieties of homoscedastic mixtures of other location families. In the case of Example \ref{ex:Laplace}, we saw that Gaussian moments and Laplacian moments coincide up to $d=3$. This means that Theorem \ref{thm:degree3} applies verbatim to homoscedastic mixtures of Laplace distributions. 
\bigskip \bigskip 

\noindent
{\bf Acknowledgments.} The authors are grateful to the Max Planck Institute for Mathematics in the Sciences, Leipzig and the Institute for Computational and Experimental Research in Mathematics in Providence, RI for facilitating discussions about this work. Carlos Am\'endola was partially supported by the Deutsche Forschungsgemeinschaft (DFG) in the context of the Emmy Noether junior research group KR 4512/1-1. 
We thank anonymous referees for suggestions to improve the presentation.
%\bibliography{homsec}
%\bibliographystyle{plain}	

\bigskip \bigskip

\noindent
\footnotesize {\bf Authors' addresses:} \hfill (corresponding Tel: +49 89 289 17436 , Fax: +49 89 289 17435)

\smallskip 

\noindent
Humboldt University, Unter den Linden 6, 10099 Berlin, Germany,
\hfill {\tt daniele.agostini@math.hu-berlin.de}

\noindent \Letter Technical University of Munich, Boltzmannstra\ss e 3, 85748 Garching, Germany,
\hfill {\tt carlos.amendola@tum.de}
%Tel: +49 89 289 17436 , Fax: +49 89 289 17435
\noindent University of Oslo, Postboks 1053 Blindern, 0316 Oslo, Norway, \hfill  {\tt ranestad@math.uio.no}

\end{document}